\documentclass[11pt]{article}
\usepackage[latin1]{inputenc}
\usepackage{amsmath,amssymb}
\usepackage{xcolor}
\usepackage{multicol}
\usepackage{ulem}
\usepackage{latexsym,epsfig}
\usepackage{amssymb,amsmath,amsfonts,amscd}
\usepackage{ifthen}
\usepackage{lscape}

\usepackage{todonotes}

\newcommand{\cm}{\noindent{\bf Proof.} }

\newtheorem{theo}{\textbf{Theorem}\ }
[section]
\newtheorem{lemma}[theo]{\textbf{Lemma}\ }
\newtheorem{criterion}[theo]{\textbf{Criterion}\ }

\newtheorem{defi}[theo]{\textbf{Definition}\ }
\newtheorem{prop}[theo]{\textbf{Proposition}\ }

\newtheorem{remark}[theo]{\textbf{Remark}\ }

\newtheorem{fact}[theo]{\textbf{Fact}\ }
\textwidth=16cm\def \fin{\hfill{$\Box$}}
 
\newenvironment{proof}{\par \cm}{\fin \par} 
 
    \newcommand{\cB}{{\mathcal{B}}}      
\newcommand{\cM}{{\mathcal{M}}}\newcommand{\cW}{{\mathcal{W}}}

\newcommand{\RR}{\mathbb{R}}\newcommand{\EE}{\mathbb{E}}
\newcommand{\PP}{\mathbb{P}}
\renewcommand{\SS}{\mathbb{S}}
\newcommand{\NN}{\mathbb{N}}
\newcommand{\XX}{\mathbb{X}}\newcommand{\ZZ}{\mathbb{Z}}

\newcommand{\Pb}{{\overline{P}}}

\newcommand{\mub}{{\overline{\mu}}}

\newcommand{\supp}{{\mathrm{supp}}}

\newcommand{\Pd }{  {\bf P} (\RR^d) }
\newcommand{\Rpm}{\RR^d_{\pm}}
\newcommand{\XN }{  \mathcal{X}_N }
\newcommand{\Sd }{  \SS^{d-1}  }

\date{}

\begin{document}

\title{
{\bf Recurrence of  
 multidimensional affine recursions  
  in the critical case}  
 }
\author{
R. Aoun \footnote{
Laboratoire d'Analyse et de Math\'ematiques Appliqu\'ees UMR 8050,
Universit\'e Gustave Eiffel, Universit\'e Paris-Est Cr\'eteil, CNRS France. \quad  richard.aoun@univ-eiffel.fr
} ,
S. Brofferio\footnote{
Laboratoire d'Analyse et de Math\'ematiques Appliqu\'ees UMR 8050,
Universit\'e Paris-Est Cr\'eteil, Universit\'e Gustave Eiffel, CNRS France. \ sara.brofferio@u-pec.fr}\ \ 
			\ \& \ M. Peign\'e\footnote{Email: Institut Denis Poisson UMR 7013,  Universit\'e de Tours, Universit\'e d'Orl\'eans, CNRS  France. marc.peigne@univ-tours.fr}}
\maketitle

%

\begin{abstract}
   We prove, under different natural hypotheses,  that the random multidimensional affine recursion $X_n=A_nX_{n-1}+B_n\in\RR^d, n \geq 1,$ is recurrent in the critical case. In particular we cover the cases where the matrices $A_n$ are similarities, invertible, rank 1 or with non negative coefficients.
      	These results are a consequence of a  criterion  of recurrence for a large class of affine recursions on $\mathbb R^d$, based on some moment  assumptions of the so-called ``reverse norm control random variable".
\end{abstract}

\noindent Keywords:   Affine recursion, recurrence, ladder epoch of a random walk, Markov walks

\noindent 
MSC 2020: 60J10, 60J05, 60B15.
 
\setcounter{tocdepth}{1}


\newpage
\section{Introduction}


We fix $d\geq 1$ and endow $\mathbb R^d$ with the euclidean norm $\vert \cdot \vert $ defined by   $  \displaystyle \vert x\vert := \sqrt{\sum_{i=1}^d  x_i^2}$ for any column vector $x=(x_i)_{1\leq i \leq d}$.

We  note   $\mathcal M(d, \mathbb R)$ the semi-group of $d\times d$-matrices with real valued coefficients. For any $A \in \mathcal M(d, \mathbb R)$ and $x \in \mathbb R^d$, we note $Ax$ the image of $x$ by the linear action of $A$ and   $\Vert A\Vert$  the norm of the matrix $A \in \mathcal M(d, \mathbb R)$ defined by 
$\Vert A\Vert= \sup_{ x \in \mathbb R^d, \vert x\vert =1 }\vert Ax\vert$.

 Let $g_n = (A_n, B_n),n = 1,2,...$ be independent and identically distributed random variables defined on the probability space  $(\Omega, \mathcal T, \mathbb P)$ with   distribution $\mu$ on $\mathcal M(d, \mathbb R)\times \mathbb R^d$. The distribution of the matrices $A_n$ is denoted $\bar \mu$.

Let $(X_n)_{n \geq 0}$ be the ``affine recursion" on $\mathbb R^d$  defined inductively by : for any $n \geq 0$
\begin{equation}\label{eq_def_X_n}
	X_{n+1} = A_{n+1}X_n+B_{n+1}.
\end{equation}
The process  $(X_n)_{n \geq 0}$ is  a Markov chain   on $\mathbb R^d$ and we are interested  in its recurrence properties in the ``critical case'', that is 
when the Lyapunov exponent  $\displaystyle \gamma_{\bar \mu}:=\lim_{n\to\infty}\frac{\mathbb E(\ln \Vert A_n\cdots A_1\Vert)}{n}$ 
 of the  product of  the random matrices $A_k$ equals $0$.

When $X_0=x$ for some fixed $x \in \mathbb R^d$, we set  $X_n= X_n^x$.  We say that the process $(X_n)_{n \geq 0}$  is {\it (topologically) 
 recurrent } on $\mathbb R^d$ when  there exists $K>0$ such that, for any $x \in \mathbb R^d$,  
\[
\mathbb P(\liminf_{n \to +\infty} \vert X_n^x\vert \leq K)=1.
\]

In this article, we propose  a general criterion which ensures that $(X_n)_{n \geq 0}$ is  recurrent  on $\mathbb R^d$. Then we apply it to four interesting situations, depending on whether the matrices $A_n$ 
  are similarities,  rank 1, invertible   or non negative matrices, with some restrictive hypotheses  in each situation   which appear classically in the theory of product of random matrices and are  labeled respectively  {\bf S}, {\bf Rk1}, {\bf I} and {\bf Nn} (see Section \ref{examples} for the precise lists).   
      Our main statement is:   
   \begin{theo}\label{thm-main} 
  	Assume that 
  	\begin{enumerate}
  		\item  Hypotheses  {\bf S}, {\bf Rk1}, {\bf I} or  {\bf Nn}  hold for the random matrices $A_n$; 
  		\item   $\mathbb E\left((\ln^+\vert B_1\vert)^{\gamma}\right)<+\infty $ for some $\gamma >2$. 
  	\end{enumerate}
  	Then,  for any $x \in \mathbb R^d$,   the Markov chain $(X_n^x)_{n \geq 0}$ is  recurrent.
  \end{theo}
 In the case {\bf I}, by using the action of the $A_n$ on some suitable exterior product $\wedge^r \mathbb R^d, 1\leq r \leq d$ of $\mathbb R^d$, we may relax our assumptions and avoid the so-called hypothesis {\bf I}-{\it Proximality}. We refer to Theorem \ref{exterior} for details; the case of non-proximal invertible matrices requires some technical adaptations. We thought it would be interesting to deal first with the case where the proximality assumption is satisfied,  proving first a criterion   based on the properties of the action of products of the matrices $A_n$ on the projective space $\Pd$  and using   classical results on products of random matrices. 

  The   stochastic
recurrence equation
$X_{n+1}= a_{n+1} X_n + b_{n+1}$
on $\mathbb R$, where the $(a_n, b_n) _{n \geq 1} $ are independent and identically distributed   random variables with values in $\mathbb R_{*+}\times  \mathbb R$, 
has been extensively studied, with special attention given to the existence of  an invariant measure and its properties,    especially its tails. This   ``random coefficients  autoregressive process'' occurs in different domains, in particular in economics,  and has been  studied intensively for several decades.  
We refer to the book by D. Buraczewski, E. Damek  \& T. Mikosch \cite{BDM}  for a general survey of the topic  and references therein.  A rich literature covers the so called 
``contractive  case"    $\mathbb E[\ln a_1]<0$; this condition   ensures that, when $\mathbb E[\ln ^+\vert b_1\vert]<{+\infty}$,  there exists on $\mathbb R$  a unique invariant  probability measure for the process $(X_n)_{n \geq 0}$.   
  In the critical case  $\mathbb E[\ln a_1]=0$, the  affine recursion  is not positive recurrent \cite{BP92Ann}. Nevertheless, at the end of the '90,  P.  Bougerol \& L. Elie \cite{BE}   (see also \cite{BBE})   proved  that  in this case there exists on $\mathbb R$ an invariant and  unbounded Radon measure  $m$ for   $(X_n)_{n \geq 0}$ ; furthermore, if $\mathbb E((\ln a_1)^{2+\epsilon} )<+\infty$, the process   $(X_n)_n$  is  topologically recurrent  on $\mathbb R$. In a recent paper   \cite{AI23},  G.Alsmeyer \& A.Iksanov  investigate the sharpness of these moment conditions.

 The affine recursion $(X_n)_{n \geq 0}$ has also been considered in dimension $d\geq 2$. The  random variables $a_n$ and $b_n$ above are replaced respectively  by $d\times d$ random    matrices $A_n$ with real entries and random vectors $B_n$ in $\mathbb R^d$.  
  The equation (\ref{eq_def_X_n}) yields, for any $n \geq 0$,  
  $$
  X_{n}= A_n\cdots A_1 X_0+  \sum_{k=1}^n A_n \cdots A_{k+1} B_k.$$
  
  The behavior of the chain $(X_n)_{n \geq 0}$ is thus related to the linear action of  the  left products  $A_{n, k}:= A_n \cdots A_{k}, 1\leq k \leq n,$ of the matrices $A_i, i \geq 1$. A huge literature is devoted to these product of random matrices, mainly in the case where the $A_k$ are either invertible or non negative;  we refer here to  \cite{BL},  \cite {FK}, \cite{G}, \cite{ H} and  \cite{L}    with references therein. For $x \in \mathbb R^d, x \neq 0,$ the study of the vector $A_{n, 1} x$   is carried out via   that of the behavior  of the  directions $A_{n, 1} x/\vert A_{n, 1} x\vert, n \geq 0,  $ when $A_{n, 1} x\neq 0$ and   that of the norm $\vert A_{n, 1} x\vert$. 
  
 On the one hand,   the  control of the  projective action of the  matrices  $A_{n, k}$   relies on some contraction properties of the closed semi-group $T_{\bar \mu}$  generated by the support of  $\bar \mu$. This requires some restrictive assumptions :  for instance,    {\it proximality} and {\it strong irreducibility property} \footnote{More details on these notions are given in section 3}  in the case of products of invertible matrices  \cite{BL}, or  existence of   matrices with positive entries  when the $A_k$ are non negative   \cite{H}. On the other hand, the study of the norm  $ \vert A  _{n, 1}x \vert$  relies on  the   cocycle decomposition   of its logarithm   as   
$$
 S_n(x):= \ln  \vert A  _{n, 1}x \vert=
\ln \vert A_1 x\vert + \ln \left\vert A_2 \left({A_1x\over  \vert A_1x\vert}\right)\right\vert+ \ldots +
\ln \left\vert A_n \left({A_{n-1, 1}x\over  \vert A_{n-1, 1}x\vert}\right)\right\vert,
$$
which can be done as soon as $  A_{k, 1} x\neq 0$ for $k\geq 1$. The  process  $(\ln  \vert A  _{n, 1}x \vert)_{n \geq 1}$  is a fundamental, even iconic, example of Markov walks on $\mathbb R$; the degree of dependence between the  increments of this sum is controlled  by  the Markov chain $(A_{n, 1} x/\vert A_{n, 1} x\vert )_{n \geq 0}$, when it is well defined.

When $d\geq 2$, the  contractive case for the affine recursion corresponds  to the case  where the Lyapunov exponent $\gamma_{\bar \mu}$ associated with  the random matrices $ A_n$ is negative. The existence and uniqueness of an invariant probability  measure is obtained following the same strategy as the one developed in dimension 1 \cite{Bra86, BP92Ann}. Various properties of this measure  have been obtained in this case, based on results  of product of random matrices (see   \cite{BDM}, chap. 4 and references therein). As far as we know,  the existence and uniqueness of an invariant Radon measure  in the  critical case $\gamma_{\bar \mu}=0$  has been investigated only in the case of  non negative matrices   :  the existence and uniqueness  of an (infinite) invariant Radon measure is established  in  \cite{BPP}  when all the ratios $A_n(i, j)/A_n(k, \ell)$ are supposed to be  bounded from above and below, uniformly in the alea $\omega$ and $i, j, k, \ell$ in $\{1, \ldots, d\}$. This ensures that the ratio $ \Vert A_{n, 1}\Vert / \vert A_{n, 1} x\vert$ are themselves $\mathbb P$-a.s. bounded from above, uniformly in $n$, in $x$ (where $x$ is an unit vectors  with positive entries) and in the alea $\omega$. This is a quite strong hypothesis and one of the   aims of the present paper is to relax it. The main new ingredient is  the so-called {\it reverse norm control} random variable    (see (\ref{RNCrandomvariable})  for a precise definition)  which allows to compare in a more flexible way  the behavior of the two sequences $( \Vert A_{n, 1}\Vert)_{n \geq 1}$ and $( \vert A_{n, 1} x\vert)_{n \geq 0}$.

In subsection 1.1,  we fix the notations.    Section \ref{sectionrecurrence} is devoted to a general criterion for  recurrence  of $(X_n)_{n \geq 0}$;  this criterion is based on the properties of the action of products of the matrices $A_n$ on the projective space $\Pd$. In Section \ref{examples} we analyse and  establish these properties in the different models we consider in this paper : the An are successively similarities, rank 1 matrices, invertible strongly irreducible and proximal or non negative matrices. Section \ref{sectionfluctuationsRk1} is devoted to the study of the fluctuations of the process  $(\ln |A_{n, 1}v|)_{n \geq 1 } $ in the rank one case.
In Section \ref{appendix} we explain  how, in the invertible case,  one  can still obtain the recurrence of the process without  the proximality assumption of the action of $A_n$ on $\Pd$; this requires  a slight extension of the main result of Section  \ref{sectionrecurrence} (see Criterion \ref{thm-from-RCN-to-cons-rep}).

\subsection{ Tools and notations}\label{toolsandnotations}
The  recurrence  of $(X_n)_{n \geq 0}$ is closely related to the contraction properties of elements of the semi-group generated by the support of $ \bar \mu  $. Assume    $\mathbb E(\ln^+ \Vert A_1\Vert)<+\infty$ and  let $\gamma_{\bar \mu}$ be   the Lyapunov exponent  of the random walk  $(A_{n, 1})_{n \geq 1}$:
\[
\gamma_{\bar \mu}:= \lim_{n \to +\infty} {1\over n} \mathbb E( \ln \Vert A_n\cdots A_1\Vert)\in[-\infty,+\infty).
\]
Kingman's subadditive ergodic theorem ensures that this limit holds almost surely
\begin{equation}\label{eq-lyap}
	\gamma_{\bar \mu}:= \lim_{n \to +\infty} {1\over n}  \ln \Vert A_n\cdots A_1\Vert\quad\PP-\mbox{a.s.}
\end{equation}
From now on,  we assume $\gamma_{ \bar \mu  }=0$.

Let ${\bf P}(\mathbb R^d)$ denote the projective space of $\mathbb R^d$,  that is the quotient space $\mathbb R^{d}\setminus  \{0\}\slash \sim$ where $x\sim y$ means that there exists $\lambda \neq 0$ such that $y= \lambda x$. We denote by $\bar x\in\Pd$ the class of the vector $x\in\RR^d$.

 For any non nul vectors $x$ and $y$ in $\mathbb R^d$, we note $x \wedge y$ their exterior product. The space ${\bf P}(\mathbb R^d)$ is endowed with the distance $\delta$ defined by: for any $ \bar x, \bar y \in  {\bf P}(\mathbb R^d)$,  
 $$\delta(\bar x, \bar y):=
\frac{ \vert x\wedge y\vert } {\vert x\vert \ \vert y\vert}=\vert \sin(\widehat{xy})\vert $$ where $x$ and $y$ are (any) representatives of $\bar x$ and $\bar y$ respectively and $\widehat{xy}$  the angle  between  $x$ and $y$. 

The projective action of an  invertible matrix $A$ on ${\bf P}(\mathbb R^d)$ is defined by
$A\cdot \bar x =  \overline{Ax}$ for any $\bar x \in{\bf P}(\mathbb R^d)$ and any representative $x \neq 0$ of the element $\bar x$.  
Since the matrices we consider here are not necessarily invertible, it is convenient to  introduce  the space $\mathcal X:={\bf P}(\RR^d)\cup \{0\}$ 
and  to define the ``projective" action of elements of $\mathcal M(d, \mathbb R)$ on $\mathcal X$ as follows: for any $A \in \mathcal M(d, \mathbb R)$, 
  $  \ A\cdot 0=0$ and,  for any  $\bar x \in{\bf P}(\mathbb R^d)$,
\[
A\cdot \overline{x}= \left\{
\begin{array}{cll} 
	 \overline{Ax}& {\rm when} & A x \neq 0,\\
	0& {\rm when} & Ax=0,
\end{array}\right.\] where
 $x\neq 0 $ is  an  arbitrary representative  of $\bar x$.

Observe that the action of $\cM(d,\RR)$ on $\mathcal X$ is measurable but  not continuous.

 In order to analyze the linear action of  $\cM(d,\RR)$ on $\RR^d$, we also consider  the quantity $\vert A\bar x\vert$ defined by: 
 \[
\vert A \overline{x}\vert = \left\{
\begin{array}{cll} 
	 \displaystyle {\vert Ax\vert \over \vert x\vert}& {\rm when} &   \bar x \in{\bf P}(\mathbb R^d) \quad (\text{where $x\neq 0$ is an arbitrary representative of}\ \bar x),\\
	0& {\rm when} &  \bar x=0.
\end{array}\right.
\]
Notice that this quantity is well defined since $ \displaystyle {\vert Ax\vert \over \vert x\vert}$ does not depend on the representative $x \neq 0$ of $\bar x$, when $ \bar x \in{\bf P}(\mathbb R^d)$.

 Let $g_n = (A_n, B_n),n = 1,2,...$ be independent and identically distributed random variables defined on the probability space  $(\Omega, \mathcal T, \mathbb P)$ with   distribution $\mu$ on $\mathcal M(d, \mathbb R)\times \mathbb R^d$.

 Let $(\overline{V}_n)_{n \geq 0}$ be the process on  $\mathcal X$ defined  by \[
\overline{V}_{n+1} = A_{n+1}\cdot \overline{V}_n
\] for any $n \geq 0$.  
This process is a Markov chain on $\mathcal X$ with transition probability kernel $\overline P$ defined by: for any  bounded Borel function $\varphi: \mathcal X\to \mathbb R $  and any $\bar v \in \mathcal X$, 
\[
\Pb\varphi(\bar v)= \int _{\mathcal M(d, \mathbb R)} \varphi(A\cdot \bar v)  \bar \mu  ({\rm d}A)
\]
where $\bar \mu$ is the projection on $\cM(d, \RR)$ of the probability measure $\mu $ on $\cM(d,\RR)\times \RR^d.$

Under several type of conditions,  it is possible to assume that  the Markov chain $(\overline{V}_n)_{n \geq 0}$ admits at least one invariant measure $\nu$ on $\mathcal X$ whose  support is contained in ${\bf P}(\RR^d)$: it holds for instance when the random matrices $A_n$ are invertible or when their entries are all positive. We explore here also other situations. 
Further criteria for the uniqueness of $\nu$ do exist, we refer to \cite{BL} and references therein.

{\bf Throughout the  paper, we fix an invariant measure $\nu$ for the Markov chain $(\overline{V}_n)_{n \geq 0}$  with support included in  ${\bf P}(\RR^d)$ and  assume that the distribution of $\overline{V}_0$ equals $\nu$  and that $\overline{V}_0$ is independent of the sequence $(A_n, B_n)_{n \geq 1}$.}

 For any $\bar v \in \Pd$ and $n \geq 1$,  it holds $\vert A_{n,1}\bar v\vert \leq \Vert A_{n, 1} \Vert$.  In order to control more precisely  the behavior of the sequences $(\vert A_n\cdots A_1\bar v\vert)_{n\geq1}$, we introduce the following random quantity, that we call the {\it reverse norm control} (RNC) random variable, defined by: for any  $\bar v \in {\bf P}(\RR^d)$ and $k \geq 1$, 
\begin{equation}\label{RNCrandomvariable}
C_{ k}(\bar v):= \sup_{n \geq k} {\Vert A_n\cdots A_{k}\Vert \over 
\vert A_n\cdots A_{k} \bar v\vert }\in [1, +\infty]
\end{equation}
with the convention ${0\over 0} = 1$ and ${c\over 0}=+\infty$ for any $c>0$. Notice that 
	$$\vert A_n \cdots A_k \bar v\vert \leq \Vert A_n \cdots A_k\Vert \leq C_{ k}(\bar v) \  \vert A_n \cdots A_k \bar v\vert.$$

  In several classical situations, this RNC random variable is $\mathbb P$-a.s. finite. For instance, by \cite[Proposition III.3.2]{BL} and   under quite general assumptions,  this is the case  when the $A_n$ are invertible ; this   property also   holds for  non negative matrices $A_n$  when $v$ has positive entries.

 We end this paragraph with the following definition which plays an important role in the sequel of the present paper. 

 \begin{defi}  Let $\nu$ be an invariant probability measure for the Markov chain $(\overline{V}_n)_{n \geq 0}$ on ${\bf P}(\RR^d)$ and assume that $\overline{V}_0$ has distribution $\nu$. 
 
 One says that the sequence $(A_n)_{n \geq 1}$ satisfies the  \underline{reverse norm control property of order $\beta >0$}  relatively to the measure  $\nu$   if  
 $$
 \mathbb E\left((\ln^+ C_{ 1}(\overline{V}_0))^\beta\right)<+\infty.
 $$
 \end{defi}

\section{A criterion for  recurrence} \label{sectionrecurrence}
The following classical result guarantees that the affine recursion $(X_n)_{n \geq 0}$ is recurrent in the contractive case. 
\begin{fact} \label{fact-cont-case}Assume $\gamma_{\bar \mu} <0$  and $\mathbb E(\ln^+\vert B_1\vert)<+\infty$. Then,   the Markov chain  $(X_n)_{n \geq 1}$  is recurrent on $\mathbb R^d$.
\end{fact} 
For reader convenience we present here a short proof.
\begin{proof} 	Observe that for any $m\in \NN$ and $x,y\in\RR^d$,
	\begin{align*}
		\lim_{n \to +\infty}\vert g_n\circ\cdots \circ g_1(x) -g_n\circ\cdots \circ g_m(y)\vert &= \lim_{n \to +\infty}\vert A_n\cdots A_m (X_{m-1}^x-y)\vert \\&\leq  \lim_{n \to +\infty}\Vert A_n\cdots A_m \Vert\  \vert X_{m-1}^x-y\vert  
		=0 \quad \mathbb P-a.s.
	\end{align*}
	Hence, the random variable $L:=\liminf_{n \to +\infty}\vert X_n^x\vert \in[0,+\infty]$ does not depend $x$; furthermore, for all $m\in\NN$,  it is independent of  the $\sigma$-algebra $\mathcal F_m$ generated by the variables  $g_1, \ \ldots \,  g_m$. By Kolmogorov's  $0-1$ law,  this random variable $L$ is constant $\mathbb P$-a.s. 
	To prove the recurrence of $(X_n^x)_{n \geq 0}$, one just needs to check  that $L$ is not infinite with positive probability. 
	
	 By an easy induction using the definition of the affine recursion (\ref{eq_def_X_n}), we may write:  
	for any $n \geq 1$ and $x \in \mathbb R^d$,
	\[
	X_n^x= A_{n, 1} x +B_{n, 1}
	\]
  where $
		A_{n, 1}=A_n\cdots A_1$ and   $B_{n, 1} =\displaystyle \sum_{k=1}^n A_n \cdots A_{k+1} B_k$  (with the convention  $A_n \cdots A_{n+1} =I$). 
	
	Since $ \EE(\ln^+\vert B_1\vert)<+\infty$ and  $\gamma_{\bar \mu}<0$,  by (\ref{eq-lyap})  it holds
	\begin{equation} \label{kjzqbchsdvkz}
		\limsup _{k \to +\infty}\vert A_1\cdots A_{k-1} B_{k}\vert^{1/k}= e^{\gamma_{\bar \mu}}<1 \qquad \mathbb P\text{-a.s.}
	\end{equation}
	 Consequently,   the series $(\overrightarrow{X}_n)_{n \geq 0}$, with   $\overrightarrow{X}_n:=\displaystyle  \sum_{k = 1}^{n} A_1\cdots A_{k-1} B_{k}$ (corresponding to the  action of the right products $g_1\cdots g_n$  of  the transformations $g_k$), is $\mathbb P$-a.s. absolutely convergent to some random variable $\overrightarrow{X}_\infty$ with values  in $\mathbb R^d$.
	Since the random variables $\overrightarrow{X}_n$ and $X_n^0=B_{n,1} $ have the same distribution,
 it yields, for $K>0$,
	\begin{align*}
		\PP(\liminf_{n\to+\infty}\vert X_{n}^0\vert \leq K) &
		\geq \PP\left(\bigcap_{N\geq 1} \bigcup_{n\geq N}(\vert X_{n}^0\vert \leq K)\right) 
		\\
		&=\lim_{N\to +\infty }\PP\left( \bigcup_{n\geq N}(\vert X_{n}^0\vert \leq K)\right) \\
		&\geq  \limsup_{N\to +\infty }\PP( \vert X_N^0\vert \leq K)=  \limsup_{N\to +\infty }\PP( \vert \overrightarrow{X}_N\vert \leq K)= \PP(\vert \overrightarrow{X}_\infty\vert \leq K) 
	\end{align*}
	 
as soon as $K$ does not belong to the denumerable set of atoms of the distribution of $\vert \overrightarrow{X}_\infty\vert $.  
For such a $K$ large enough, the last term is   positive,  hence  $ \displaystyle \liminf_{n\to+\infty}\vert X_{n}^0\vert <+\infty \ \mathbb P$-a.s.

\end{proof}

{\bf Throughout this  paper,  we focus on the case when the Lyapunov exponent $\gamma_{\bar \mu}$  equals 0}.

In the critical case, i.e. $\gamma_{\bar \mu}= 0$, this argument no longer works and  and should be modified accordingly.

From now on, we  fix   $\rho \in (0, 1)$ and, for any  $\bar v \in \Pd$,  we consider the random variable $ \ell_1^{(\rho)}(\bar v)$ defined by
$$ 
 \ell_1^{(\rho)}(\bar v):= \inf \{n \geq 1 \mid  \vert A_n \cdots A_1\bar v\vert \leq \rho\} \in \mathbb N \cup \{+\infty\}.
$$
Notice that  the  $ \ell_1^{(\rho)}(\bar v)$  are stopping times with respect to the   natural filtration associated with the sequence $(A_n)_{n \geq1}$.

Our main criterion is the following statement.
\begin{criterion}\label{thm-from-RCN-to-cons}
  Let $\nu$ be a $\mub$-invariant probability measure on $\Pd$  and   $ \overline{V}_0$   a random variable  with distribution   $\nu$ and independent of the sequence $(A_n, B_n)_{n \geq 1}$. 		
Assume that   there exist   constants  $\alpha >0$, $   \beta >1 + {1\over \alpha}$  and $\gamma\geq \max\{{1\over \alpha},1\} $ such that the following hypotheses hold:

 ${\bf A}_1(\alpha)$-$\quad 
\mathbb E\left(\left(\ell_1^{(\rho)}(\overline{V}_0)\right)^\alpha\right)<+\infty $ for any $\rho \in (0, 1)$.

 ${\bf A}_2 (\beta) $-$  \quad   
 \mathbb E\left((\ln^+ C_{ 1}(\overline{V}_0))^\beta\right)<+\infty. \quad $ {\it (Reverse norm control property)}

 ${\bf B} (\gamma)$-$\quad   
\mathbb E\left((\ln^+\vert B_1\vert)^{\gamma}\right)<+\infty.
$

\noindent Then, for any $x \in \mathbb R^d$,   the chain $(X_n^x)_{n \geq 0}$ is recurrent.
\end{criterion}

 We will see that typically  $A_1(\alpha)$ hold for any $\alpha< {1\over 2}$. Thus  recurrence follows when  $\beta>3$ and $\gamma>2$.

The random  variables $C_1(\bar v)$ are really relevant in dimension $d \geq 2$. Indeed,  in dimension $d=1$, the variable $C_1(\bar v)$  equals 1.

\noindent {\bf Proof of criterion \ref{thm-from-RCN-to-cons}}.
We follow the  strategy  developed by L. Elie in \cite{E}. 
Fix $\rho\in (0,1)$ such that $\ln \rho <- \EE \left(\ln  C_{ 1}(\overline{V}_0)\right)$, this is possible  since $\beta>1$ and  hypothesis ${\bf A}_2 (\beta)$ yields  $\mathbb E(\ln^+ C_{ 1}(\overline{V}_0))<+\infty$.

Throughout the present proof, let $(\overline{V}^{(k)})_{k \geq 0}$ be a sequence of independent and identically distributed random variables  with distribution $\nu$ and assume that this sequence is independent of $(A_n, B_n)_{n \geq 1}$. We set $\ell_0=0$
and, for any $k \geq 1$, 
\[
\ell_k =  \inf\{n \geq \ell_{k-1}+1: \vert A_n \cdots A_{\ell_{k-1}+1}  \overline{V}^{(k-1)} \vert \leq \rho\}. 
\]
 Without loss of generality, we may assume $\overline{V}^{(0)}= \overline{V}_0$; let us emphasize that the random variables $\overline{V}_k= A_k\cdots A_1\cdot \overline{V}_0, k \geq 1,$ have also  distribution $\nu$ but that the sequence $(\overline{V}_k)_{k \geq 0}$ is not i.i.d. The interest of the sequence $(\overline{V}^{(k)})_{k \geq 0}$ is to decompose the product $A_{\ell_n}\cdots A_1$ as a product of i.i.d. random variables as explained below. 

Observe that   the sub-process $(X_{\ell_n}^x)_{n \geq 0}$ is a Stochastic Dynamical System defined recursively by   affine transformations. More precisely,
$$
X_{\ell_n}^x=  \widetilde{g}_n \circ \cdots \circ \widetilde{g}_1(x)
$$
with  
$$
\widetilde{g}_k=(\widetilde{A_k}, \widetilde{B_k}):=g_{\ell_{k}}\circ\dots\circ g_{\ell_{k-1}+1}=(A_{\ell_{k}}\dots A_{\ell_{k-1}+1},\sum_{i= \ell_{k-1}+1}^{\ell_{k}} A_{\ell_{k}} \cdots A_{i+1} B_i).
$$
The  random variables $ \widetilde{g}_k, k \geq 1,$ are independent and identically  distributed,     with distribution   $\widetilde{\mu}={\rm dist} (\widetilde{g}_1)$ on $\cM(d,\RR)\times \RR^d$.

Let us now check that  the sub-process $(X_{\ell_n})_{n \geq 0}$ satisfies the hypotheses of Fact \ref{fact-cont-case}. This will prove that  $(X_{\ell_n})_{n \geq 0}$, hence $(X_n)_{n \geq 0}$, is recurrent.

We set $\alpha_k:= \Vert A_{\ell_k }\cdots A_{\ell_{k-1}+1}\Vert$. By construction, the random variables $\alpha_k, k \geq 1,$  are independent and identically distributed. We claim that they are log-integrable and have   negative log-mean.    In fact by definition of $C_{ 1}(\overline{V}_0)$, it holds 
$$
\Vert A_{\ell_1} \cdots A_1\Vert \leq C_{ 1}(\overline{V}_0) \vert A_{\ell_1} \cdots A_1 \overline{V}_0 \vert 
\leq \rho  \ C_{ 1}(\overline{V}_0)
$$ 
so that
\[
\mathbb E(\ln \alpha_1)\leq \ln \rho+ \EE (\ln  C_{ 1}(\overline{V}_0))=\ln \rho+ \EE \left(\ln  C_{ 1}(\overline{V}_0)\right)<0.
\]
Hence,  the Lyapunov exponent of $\widetilde{\mu}$ satisfies
$$\gamma_{\widetilde \mu}= \lim_{n \to +\infty} {1\over n}  \ln \Vert A_{\ell_n, 1}\Vert\leq \lim_{n \to +\infty} \frac{\ln \alpha_1+\cdots + \ln \alpha_n}{n}\leq  \ln \rho+ \EE \left(\ln  C_{ 1}(\overline{V}_0)\right)<0.$$

In order to check that $\widetilde{B_1}=B_{\ell_1,1}$ is log-integrable, we observe that 
\begin{align*}
\vert B_{\ell_1,1}\vert  &\leq
\sum_{i={1}}^{\ell_1} \Vert A_{\ell_1}\cdots A_{i+1}\Vert  \vert B_i\vert \leq \sum_{i=1}^{\ell_1} C_{i+1}(\overline{V}_i) \vert B_i\vert \quad \mathbb P\text{-a.s.}
\end{align*}
Indeed, on the one hand,  the inequalities    $\vert A_{\ell_1}\cdots A_{1}\overline{V}_0\vert\leq \rho$ and $ \vert A_{i}\cdots A_{1}\overline{V}_0\vert >\rho >0$  for $ 1\leq i < \ell_1 $ yield   
 \begin{align*}
 	\Vert A_{\ell_1}\cdots A_{i+1}\Vert
 	&\leq 
 	C_{ i+1}(\overline{V}_i)\vert A_{\ell_1}\cdots A_{i+1} \overline{V}_i \vert
 	\\
 	&= 
 	C_{ i+1}(\overline{V}_i){\vert A_{\ell_1}\cdots A_{1} \overline{V}_0 \vert \over \vert A_{i}\cdots A_{1} \overline{V}_0 \vert}\leq 	C_{ i+1}(\overline{V}_i){\rho \over \rho }= 	C_{ i+1}(\overline{V}_i)
 \end{align*}
 for $1\leq i < \ell_1$.  On the other hand,   
 $$	\Vert A_{\ell_1}\cdots A_{\ell_1+1}\Vert=\vert I\vert =1\leq 	C_{ i+1}(\overline{V}_i).$$
 
Thus
\begin{equation}\label{zkjlbedh}
\ln^+\vert B_{\ell_1,1}\vert  \leq \ln^+\ell_1+\max_{1\leq i \leq \ell_1} \ln^+ C_{ i+1}(\overline{V}_i) +\max_{1\leq i \leq \ell_1} \ln^+ \vert B_i\vert\quad \mathbb P\text{-a.s.}
\end{equation}
Under hypothesis $A_1(\alpha)$, the random variable $ \ln^+\ell_1$ is integrable.
The fact that 
$$\EE(\max_{1\leq i \leq \ell_1} \ln^+ \vert B_i\vert)< +\infty$$ whenever $\textbf{B}(\gamma)$ is satisfied for some $\gamma \geq\max\{{1\over \alpha},1\}$ is a rather classical result that has been used by several authors (see for instance  \cite[Proposition 4]{CKW}). 
However this argument  uses deeply the fact that the $B_i$ are independent;  it cannot be applied to the sequence $(C_{ i+1}(\overline{V}_i))_{i \geq 0}$ which  is stationary but not independent.  To deal with the second term in the right hand side of  (\ref{zkjlbedh}), we need to use Lemma \ref{lemmedusup}, proved hereafter, which is more general but requires stronger moment. In particular applying this Lemma with $Y_i=  \ln^+ C_{ i+1}(\overline{V}_i)$ and $ \tau=\ell_1$, we can conclude that hypotheses ${\bf A}_1(\alpha)$ and ${\bf A}_2(\beta)$ with $\beta>1+1/\alpha$ ensure that there exists a positive constant $C(\alpha, \beta)$ s.t. 
$$\EE\left(\max_{1\leq i \leq \ell_1} \ln^+ C_{ i+1}(\overline{V}_i)\right)\leq \ C(\alpha,\beta)  \ \mathbb E(\ell_1^\alpha)^{\frac{ \beta -1}{ \beta }}\ \mathbb E  \left(\ln^+ C_1(\overline{V}_0)^\beta\right)^{1/\beta}< +\infty.$$
\rightline{$\Box$}

\begin{lemma}\label{lemmedusup}
Let $(Y_i)_{i\geq 1}$ be a sequence of non negative random variables and $\tau $ a $\mathbb N$-valued random variable, defined on a probability space $(\Omega, \mathcal F, \mathbb P)$.
  Then,  for  $ \alpha >0$ and $ \beta  >\frac{1+\alpha}{\alpha}$, there exists a constant $C= C(\alpha, \beta )$ s.t.
\begin{equation}\label{kjazeyvcg}
\mathbb E\left(\max_{1\leq i\leq \tau} Y_i \right)
\leq 
C\ \EE(\tau^\alpha)^{\frac{ \beta -1}{ \beta }}\ \sup_{i \geq 1} \mathbb E \left(  Y_i ^{   \beta }\right)^{1/ \beta }.
\end{equation}
\end{lemma}
Proof.  For any $\delta >0$,
$ \quad \displaystyle 
\left(\max_{1\leq i\leq \tau} Y_i\right) \leq  \max_{1\leq i\leq \tau} \left({Y_i \over i^\delta} \ i^\delta\right) \leq \sup_{i \geq 1}\left( {Y_i \over i^\delta}\right) \ \tau^\delta.
$

\noindent 
Hence, by H\"older inequality,
$\displaystyle 
\mathbb E\left(\left(\max_{1\leq i\leq \tau} Y_i\right) \right)
\leq 
\mathbb E\left( \sup_{i \geq 1}\left( {Y_i^{   \beta }\over i^{\delta  \beta }}\right)\right)^{1\over  \beta } \mathbb E(\tau ^{\delta q})^{1\over q} 
$
where  $q$ satisfies $\displaystyle {1\over  \beta }+{1\over q} = 1$. Hence

\begin{align*}
\mathbb E\left( \sup_{i \geq 1}\left( {Y_i^{   \beta }\over i^{\delta  \beta  }}\right)\right)
&\leq 
\mathbb E\left( \sum_{i \geq 1}  {Y_i^{   \beta  }\over i^{\delta  \beta  }} \right)\leq \sum_{i\geq 1} {\EE( Y_i^{   \beta  }) \over i^{\delta  \beta  }}\quad \leq \sup_{i \geq 1} \mathbb E(Y_i^{   \beta })\times \left(\sum_{i \geq 1} {1\over i^{\delta  \beta  }}\right)
\end{align*}
Take $\delta=\alpha/q=\alpha ( \beta -1)/  \beta $. Then $\delta  \beta = \alpha  ( \beta -1)>1$ if $ \beta >(1+\alpha)/\alpha$.
Inequality (\ref{kjazeyvcg}) follows  with  $C(\alpha,\beta)= \sum_{i \geq 1} {1\over i^{\alpha  (\beta-1)}}< +\infty$.

\rightline{$\Box$}

\section{Application to significant matrix semigroups}\label{examples}
 We present here various situations where the matrices $A_n$ belong to different and important classes of matrix semigroups: similarities, rank 1, invertible  and non negative matrices. In Subsection \ref{hypotheses}, we present   a list of hypotheses in each situations. In subsections \ref{momentladder} and \ref{RNCP}, we explain how these  different conditions imply the recurrence of the process $(X_n)_{n \geq 0}$.
 
 In all these  cases, there  exists a $\bar{\mu}$-invariant probability $\nu$ on $\Pd$  and the  affine recursion is controlled by the 1-dimensional  random walk  $(S_n)_{n \geq 0}$,   with markovian increments, defined by 
		\begin{equation}\label{eq_S_n}
			S_0=0 \quad {\rm and} \quad S_n= S_n(\overline{V}_0):=\ln \vert A_n\cdots A_1 \overline{V}_0\vert =\sum_{k=1}^n\ln\vert A_k\overline{V}_{k_-1}\vert .
	\end{equation}	
The proof of our main theorem relies in particular on  results concerning the fluctuations of this process $(S_n)_{n \geq 0}$ and on the contraction properties  on $\Pd$ of the  random  matrices $A_{n, 1},  n \geq 1$.  These results and properties require different moment  and non degeneracy    assumptions according to the explored situation; 	the formulation of these hypotheses vary from case to case, but all of them   ensure in particular that:   
	\begin{itemize}
		\item The increments of $(S_n)_{n \geq 0} $ have at least \textit{moment} of order $2$, ``from  above" and ``from below". 
		The control ``from above" can be seen in the fact  that $\ln^+ \vert A_1\vert $ has at least moment of  order  $ 2$. The one ``from below" is reflected in the conditions called ``reverse moments" which ensure that the random variable  
		$${\mathfrak n} (A_1, \overline{V}_0):=  \Vert       A_1 \Vert / \vert A_1\overline{V}_0\vert
		\ \in [1, +\infty]$$
		is sufficiently integrable; it implies in particular that   the quantity $\vert A_1\overline{V}_0\vert$ is not too close to zero.	
		\item  The Markov walk $(S_n)_{n \geq 0} $ is {\it non degenerate}, i.e. its variance  $\sigma^2:=\lim_{n\to\infty} \mathrm{Var}(S_n)/n $ is non zero. In particular,  the process $(S_n)_{n \geq 0} $ is unbounded.
		\item The projective action of the matrices $A_{n, 1}$ has nice ``contraction" properties (except in case  {\bf S},  which in some sense a $1$-dimensional situation)
		; more precisely,  there exists $\rho<\in [0, 1)$ and $c>0$ such that	 	\begin{equation} \label{eq_cont-P}
	\forall \bar u,\bar v \in \Pd, \forall n \geq 1  \qquad 		\mathbb E( \delta(A_{n, 1}\cdot \bar u, A_{n, 1}\cdot \bar v) ) \leq c\ \rho^n.
		\end{equation}
	\end{itemize}

\vspace{5mm} 

\subsection{Hypotheses}\label{hypotheses}

\subsubsection{\bf Case   {\bf S}: the $A_n$ are  similarities}

We assume here   that $A_n= a_n R_n$ with $a_n \in \mathbb R^{*+}$ and $R_n \in O(d, \mathbb R)$. 

In this case,  the uniform distribution on $\Pd$ is $\bar \mu$-invariant and $\displaystyle S_n:=\ln \vert A_n\cdots A_1 \overline{V}_0\vert =\sum_{k=1}^n\ln a_k$, for any initial direction $\overline{V}_0$. 
This is the simplest of our example since it is  controlled by the classical 1-dimensional random walk on $\mathbb R$  with i.i.d increments $\ln a_n$.

We consider the following hypotheses.

\noindent  \underline{{\it Hypotheses}  {\bf S }}

 {\underline{ {\bf S}-{\it Moments} } : $\mathbb E(( \ln a_n)^2) <+\infty$.

   {\underline{ {\bf S}-{\it Non degeneracy} } : $  V(\ln a_1) >0$.

   {\underline{ {\bf S}-{\it Centring}} :  $\mathbb E( \ln a_n)= 0$

\vspace{5mm}

\subsubsection{\bf  Case   {\bf Rk1}: the $A_n$ are rank $1$   matrices}
 
 We assume here that  there exists a  sequence $ (\tilde w_n, w_n, a_n)_{n \geq 1}$  of i.i.d. random variables with values in $\Sd\times \Sd\times \mathbb R^{*+}$ such that $\mathrm{Im}(A_n)= \RR w_n$, $\mathrm{Ker} (A_n)=\tilde{w}_n^\perp$ and $\vert A_n\vert =a_n$. In other words, 
\[\forall v \in \mathbb R^d, \quad A_nv= a_n \langle \tilde w_n, v\rangle w_ n
\]
where $\langle \cdot,\cdot\rangle$ denotes the  standard scalar product in $\RR^d$. In particular
\begin{equation}\label{eq-prod-Rk1}
	A_{n,1}v=a_n\cdots a_1 \langle \tilde w_n, w_{n-1}\rangle\cdots \langle \tilde w_1, v\rangle w_n.
\end{equation}
Suppose  that  $\PP(\langle \tilde w_{n+1},w_n\rangle=0)=0$. Then, the  distribution $\nu$ of the random  variables $\overline{w}_n$ is  a $\bar \mu$-invariant probability measure   on $\Pd$.  Indeed, let $w_0$  be  a $\Sd$-valued random variable  with distribution $\nu$ and  independent of $ (\tilde w_n, w_n, a_n)_{n \geq 1}$ ;   then  $\overline{V}_0=\overline{w}_0, 
  \overline{V}_n=A_{n,1}\cdot \overline{w_0}=\overline{w_n} $ for $n \geq 1$ and $ \overline{V}_n$    has  also distribution $\nu$.
Since the former equality holds for any starting direction $\overline{w'_0}$ such that $\PP( \langle \tilde w_1, w'_{0}\rangle=0)=0$, it also readily  implies that $\nu$ is the unique $\bar \mu$-invariant probability measure on ${\bf P}(\RR^d)$.

The 1-dimensional  random walk $(S_n)_{n \geq 0}$  which controls the behavior of the affine recursion  (\ref{eq-prod-Rk1}) can be written as,  \begin{equation}\label{eq-Sn_Rk1}
	S_n(\overline{V}_0)=\ln  \vert A_{n,1}\overline{V}_0\vert=\sum_{k = 1}^n\ln a_k\vert\langle\tilde{w}_k,w_{k-1}\rangle\vert.
\end{equation}

This is a Markov random walk with stationary increments $Y_k:=\ln a_k \vert \langle\tilde{w}_k,w_{k-1}\rangle\vert, k \geq 1$. Observe that $Y_k$ and $Y_{k+1}$ are not independent but $Y_k$ and $Y_\ell$ are independent as soon as $\vert k-\ell\vert \geq 2$.
The ergodic theorem ensures  that $ (S_n(\overline{V}_0)/n)_n$ converges $\mathbb P$-a.s. to $\EE(Y_1)$, this yields 
$$\gamma_{\bar \mu}=\EE(Y_1)=\EE(\ln |a_1\langle\tilde{w}_1,w_{0}\rangle|).$$
Notice at last that   ${\mathfrak n} (A_1, \overline{V}_0):=  \Vert       A_1 \Vert / \vert A_1\overline{V}_0\vert=1/\vert \langle\tilde{w}_1,w_0\ le\vert $.

We introduce the following assumptions.

\noindent \underline{\it Hypotheses  {\bf Rk1}  }

\underline{{\bf  Rk1}-{\it Moments}} :  {\it there exists  $\varepsilon>0$ 
	such that  }$\mathbb E\left(\vert \ln a_1 \vert^{2+\varepsilon}\right) <+\infty.$

 \underline{{\bf Rk1}-{\it Reverse moments}} : {\it there exists  $\varepsilon>0$ 
	such that  }$\mathbb E\left( \left\vert\ln \vert \langle \tilde w_1, w_0\rangle\vert \right\vert ^{3+\varepsilon}\right) <+\infty.$

\underline{{\bf  Rk1}-{\it Non degeneracy}}  :   $\sigma^2=V(Y_2) +2 {\rm cov} (Y_1, Y_2) >0$, 
 {\it with} $Y_2:= \ln a_2+\ln \vert \langle \tilde w_2,  w_1\rangle \vert $ {\it and} 
 $Y_1:= \ln a_1+\ln \vert \langle \tilde w_1,  w_0\rangle \vert.$

 \underline{{\bf Rk1}-{\it Centring}} :  $\mathbb E(\ln a_1)+\mathbb E( \ln \vert\langle\tilde w_1, w_0\rangle \vert)=0$.
}}

 At  first glance, this rank 1 situation may seem a little artificial, but it's not the case.  Indeed,  when the $A_n$ are either invertible or non negative   (see  the cases {\bf I} and {\bf Nn} below),  the closure in $\mathcal M(d, \mathbb R)$ of the semi-group   generated by the support of $\bar \mu$ contains a rank one matrix. More precisely, 
if the  Lyapunov exponent $\gamma_{\bar \mu}$ equals 0,  the sequence $(A_{n,1})_{n \geq 0}$ is recurrent in $\mathcal M(d, \mathbb R)$ and,   under standard hypothesis, its cluster points either equal $0$ or   have rank 1.  Furthermore, the  rank 1 situation explored here   is  a handy toy model  to understand how to extend results obtained  for products of invertible matrices to other  semigroups of   matrices.

}

\vspace{5mm}

\subsubsection{\bf  Case   {\bf I} : the $A_n$ are invertible}\label{sect-I}

We assume here  that the random variables  $ A_n, n \geq 1, $ are ${\rm Gl}(d,  \mathbb R)$-valued  where  ${\rm Gl}(d,  \mathbb R)$ equals the set of $d\times d$ invertible matrices with real coefficients. We assume $d\geq 2$, the case $d=1$ corresponds to a scalar affine recursion (which can be viewed as a particular case of {\bf S})  and is treated under the additional non-degeneracy condition.

Product of random  matrices in ${\rm Gl}(d, \RR)$ have been widely studied. We refer in particular to the works of  Guivarch  \cite{G} and Le Page \cite{L}   and to the books of Bougerol and Lacroix \cite{BL} or  Benoist and Quint \cite{BQ}.

Following \cite{BL}, we now introduce   the following   hypotheses:

\noindent \underline{\it Hypotheses {\bf   I }}

\underline{{\bf I}-{\it (Exponential) Moments}} :  {\it there exists  $\varepsilon>0$ 
	such that  } $\mathbb E( \Vert A_1 \Vert ^{\varepsilon}) <+\infty.$   
	 
{\underline{{\bf I}-{\it (Exponential) Reverse moments}} : {\it there exists  $\varepsilon>0$ 
	such that} 
$\mathbb E\left( \Vert A_1^{-1}\Vert ^{\varepsilon}\ \right) <+\infty.$}

 \underline{{\bf I}-{\it Proximality}} :   {\it the semigroup $T_{\overline{\mu}}$ generated by the support of $\overline{\mu}$ contains a proximal element, that is a  matrix $A$ which admits an eigenvalue $\lambda \in \mathbb R$ which has multiplicity one and whose all other eigenvalues have modulus $<\vert\lambda\vert$.} 

If $A \in T_{\overline{\mu}}$ is proximal, then the sequence   $(A^{2n}/\Vert A^{2n}\Vert)_{n \geq 1}$ converges to a rank 1 matrix. 
 		
 \underline{{\bf I}-{\it Strong irreducibility}} : {\it there exists no proper finite union of  subspaces of  $\mathbb R^d$ which  is
	invariant with respect to all elements of  $T_{\overline{\mu}}$.} 

 \underline{{\bf I}-{\it Centring}} :  $\gamma_{\bar \mu}=0$.

 Before presenting the last example considered in this article, let us make  a few comments on the assumptions {\bf I}.
	
 First,  as announced in the introduction,  notice that   by \cite[Lemma 4.13]{BQ}  assumption {\bf I}-{\it Proximality} is always satisfied when considering the action of the $A_n$ on some suitable exterior product $\wedge ^r\mathbb R^d, 1\leq r \leq d$. Therefore, this hypothesis with $r=1$ is not required to get the recurrence of ($X_n)_{n \geq 0}$. As a first step, we add it here to explain how case {\bf I} can be treated as the other cases by using criterion \ref{thm-from-RCN-to-cons}. We refer the reader to section \ref{appendix}
 for an extension of this criterion.

 Secondly, hypotheses {\bf I}-{\it (Exponential) Moments} and {\bf I}-{\it (Exponential) Reverse moments} imply that the random variable ${\mathfrak n} (A_1, \overline{V}_0):=  \Vert       A_1 \Vert / \vert A_1\overline{V}_0\vert\leq \Vert       A_1 \Vert \Vert A_1^{-1}\Vert $ has an exponential moment of order $\varepsilon/2$.

  Forty  years ago, E. le Page \cite{L} proved  that, under hypotheses {\bf I}-{\it (Exponential) Moments and reverse Moments},  {\bf I}-{\it Proximality} and {\bf I}-{\it Strong irreducibility},  for any $\bar v \in {\bf P}(\mathbb R^d)$, the process   $(S_n(v)=\ln|A_{n,1}v|)_{n}$ satisfies a Central Limit Theorem (CLT). He established in particular that the ``asymptotic" variance
 \[
 \sigma^2:= \lim_{n \to +\infty}\frac{\mathbb E(\ln \vert A_{n, 1} x\vert -n \gamma_{ \bar \mu  })^2 }{n}
 \]
 is non zero (see also \cite[Chapter V-8, Theorem 5.1]{BL}).  This corresponds to the  {\it Non degeneracy} condition  in   cases {\bf S} and {\bf Rk1}.  His proof relies on the following ``mean contraction property" :   for $\epsilon \in (0, 1 )$ small enough, there exist  $\rho_\epsilon \in (0, 1)$  and $c_\epsilon>0$ such that, for any $\bar u, \bar v \in \PP(\RR^d)$,
 \begin{equation} \label{meanI}
 	\mathbb E( \delta^\epsilon(A_{n, 1}\cdot \bar u, A_{n, 1}\cdot \bar v) ) \leq c_\epsilon\ \rho_\epsilon^n\ \delta^\epsilon (\bar u, \bar v).
 \end{equation}
 (see \cite[Chapter V-2, Proposition 2.3]{BL} ). In particular this implies (\ref{eq_cont-P}).
 
 More recent works have significantly relaxed moment hypotheses to prove similar results. In particular  Y.Benoist and J-F. Quint \cite{BQ}  proved the CLT under the optimal second moment condition  $\mathbb E((\ln  \Vert A_1\Vert)^2+(\ln (\Vert A_1^{-1}\Vert)^2)<+\infty$ and a work in progress of A. P\'eneau \cite{Pen23} should guarantee the projective contraction property (\ref{eq_cont-P}) without any  moment assumption. 
 
 In the present paper we still need exponential moments since we rely on results of \cite{GLP} for fluctuation properties of $S_n$. However, it is reasonable to think  that future developments of this theory will ensure that the second order moments are sufficient to obtain recurrence of affine recursion.

\vspace{5mm}

\subsubsection{\bf Case   {\bf Nn} : the $A_n$ are non negative }

Another interesting case that has been widely studied in the last year is the product of matrices with non negative coefficient.
In this case,  the measure $\mu$ is no longer  supported by a group but lives in  the semi-group  $\mathcal S$ of matrices     with nonnegative entries such that    each column contains at least one  positive entry.   Notice that the action  of $T_\mu$ on the full projective space is no longer continuous. However, the   semi-group $\mathcal S$ acts continuously  on the simplex  $\mathbb X$  of nonnegative vectors of norm 1; this action possesses  nice contraction properties, which yield to results similar to the one obtains in the case of  invertible matrices. From a historical point of view, this case was studied first; we refer to the works  by Furstenberg \& Kesten \cite{FK} and Hennion \cite{H}.

 The set $\mathbb X =\SS^{d-1}\cap  (\mathbb R^+)^d $ can be identified with the subset of the projective space $\Pd$  whose elements correspond to half lines in the cone $(\mathbb R^+)^d$. Observe for any $A\in \mathcal S$ and $v\in \mathbb X$, the element $A \cdot v:= A   v/\vert Av\vert$ belongs to $\mathbb X$.  

Recent papers on  random product of nonnegative matrices have proved that is useful in this setting  to endow $\RR^d$ with  the ${\rm L}^1$-norm  $\vert \cdot\vert _1$ and to provide $\XX$ with a special distance ${\mathfrak d}$ that is strongly contracted by the action of the elements of $\mathcal S$ (the reader can find  in \cite[Section 10]{H}   and \cite[Section 2.2]{BPP} 
a precise description of the properties of the distance ${\mathfrak d}$).  In the present paper, in order to  remain  consistent with the other cases, we prefer to use    the euclidean norm $\vert \cdot\vert $ on $(\mathbb R^+)^d$ and the distance $\delta$ on $\mathbb X$. We  can translate known results in the present setting, using the facts that
\begin{itemize}
	\item $ \vert x\vert \leq \vert x\vert _1\leq \sqrt{d} \vert x\vert $ for any $x \in \RR^d$;
	\item  the metrics $\delta$ and ${\mathfrak d}$ satisfies $\delta(\bar u, \bar v)\leq  \vert u-v\vert \leq   2 {\mathfrak d}(u, v )$ for any $u, v \in \mathbb X$ and induce the same
	topologies  on  $\XX$. Nevertheless, $\mathfrak d$ is not equivalent with neither other metrics $\vert \cdot \vert $ and $\delta$ on $\mathbb X$.
\end{itemize}

  For any $A =(A(i, j))_{1 \leq i, j \leq d} \in \mathcal S$,   set   $\displaystyle 
v(A) := \min_{1\leq j\leq d}\Bigl(\sum_{i=1}^d A(i, j)\Bigr). $ This quantity is of interest since $\frac{v(A)}{\sqrt{d}}\  \vert x\vert \leq \vert Ax\vert \leq \Vert         A \Vert\ \vert x\vert$  for any $x \in \mathbb R_{+}^d$, with $v(A) >0$; in particular, 
  $\mathfrak n (A,v)\leq  {\sqrt{d} \Vert A\Vert \over v(A)}.  $ 

Following \cite{Pham}, we now introduce   the following   hypotheses:

\noindent {\bf Hypotheses    Nn }

 \underline{{\bf Nn}-{\it Exponential moments}} :  {\it there exists  $\varepsilon>0$ 
	such that  }$\mathbb E\left(  \Vert A_1 \Vert  ^{\varepsilon}\right) <+\infty.$

 \underline{{\bf Nn}-{\it Reverse moments}} : {\it there exists  $\varepsilon$ 
	such that  }$\mathbb E\left( \left(\ln {\Vert A_1\Vert\over v(A_1)}\right) ^{6+\varepsilon}\right) <+\infty.$

 \underline{{\bf Nn}-{\it Contraction property }} :   {\it There exists $n_0\geq 1$ such that $\mu^{\star n_0}(\mathcal S^+)>0$. }

	 \underline{{\bf Nn}-{\it Irreducibility}} :   {\it There exists no affine subspaces $\mathcal A$ of $\mathbb R ^d$ such that $\mathcal A \cap \mathcal C$ is non empty and bounded and invariant under the action of all elements of the support of $\mu$. }

 \underline{{\bf Nn}-{\it Centring}} :  $\gamma_{\bar \mu}=0$.

The affine recursion generated by non negative matrices  has been investigated in \cite{BPP} under the restrictive condition that the $A_n$ are $\mathcal S_\delta$-valued for some $\delta>0$, where   $\mathcal S_\delta$ denotes the set of matrices $A$ in  $\mathcal S$ such that
$A(i, j)\geq \delta A(k, l)$ for any $1\leq i, j, k, l\leq d.$   Here, we break free  from this   strong assumption.

\vspace{5mm}

Let us now give a few comments on the above hypotheses.

By \cite{H}, under hypotheses {\bf Nn}-{\it Exponential moments}, {\it Reverse moments},  {\it Contraction property} and {\it Irreducibility}, the ``asymptotic" variance
\[
\sigma^2:= \lim_{n \to +\infty}\frac{\mathbb E(\ln \vert A_{n, 1} x\vert -n \gamma_{ \bar \mu  })^2}{n} 
\]
 is non zero (which corresponds to condition {\it Non degeneracy}  in the cases {\bf S} and {\bf Rk1}) and the sequence $\left({1\over \sqrt{n}}(\ln \vert A_{n, 1} x\vert-n\gamma_{\bar \mu})\right)_{n \geq 1}, x \neq 0, $  converges in  distribution  to the normal distribution $\mathcal N(0, \sigma^2)$.  
  
At last,  
under hypothesis {\bf Nn}-{\it Contraction property},  there exist $\rho  \in (0, 1)$  and $c >0$ such that, for any $u, v \in \XX$,
\begin{equation} \label{meanN}
\mathbb E( \mathfrak{d}(A_{n, 1} \cdot u, A_{n, 1} \cdot v) ) \leq c\ \rho ^n\  \mathfrak{d} (u, v).
\end{equation}
This property is of crucial interest  below (see Lemma \ref{lem-cont-RCN}).

\subsection{On the moments of the ladder time $\ell_1^{(\rho)}(\bar v)$}\label{momentladder}
We will now demonstrate that   in the above four situations, there exist positive reals $\alpha$ and $\beta$, with $\beta >1+{1\over \alpha}$ such that    conditions ${\bf A }_1(\alpha)$ and ${\bf A }_2(\beta)$  of criterion \ref{thm-from-RCN-to-cons} hold  are satisfied

 We prove here the following statement concerning fluctuations of random walks.

\begin{prop}\label{zekbghjkd} Under hypotheses  
  {\bf S},  {\bf Rk1}, {\bf I} or  {\bf Nn}, 
the  condition ${\bf A}_1(\alpha)$ holds    for any $\alpha \in (0, 1/2)$.
\end{prop}

\begin{proof} In the case {\bf S}, the above statement is a consequence of classical results on fluctuations of random walks on $\mathbb R$ for which it is well  known that, for any $\rho >0$ and $\bar v \in {\bf P}(\RR^d)$,  there exists a constant $c_\rho >0$ s.t. $\mathbb P( \ell_1^{(\rho)}(\bar v)>n)\sim  c_\rho/\sqrt{n}$ as $n \to +\infty$ \cite{F}.

Similar results do exist in the case {\bf I}   and {\bf Nn}  by    \cite{GLP} and  \cite{Pham}  respectively (with $v \in \mathbb X$ in the case {\bf Nn}).

The case {\bf Rk1} is  a consequence of a general result of fluctuations of  Markov walks on $\mathbb R$  proved in \cite{GLL};   we explain  in  Section \ref{sectionfluctuationsRk1} how  this general result can be applied to the {\bf Rk1} situation.

\end{proof}

\subsection{On the reverse  norm control property}\label{RNCP}
By the above subsection, the hypothesis ${\bf A}_1(\alpha)$ holds  for any $\alpha \in (0, 1/2)$. We  now check that, for some $\beta >1+{1\over \alpha} > 3$ ,  the RNC (= reverse norm control) property  ${\bf A}_2(\beta)$ holds  under hypotheses  {\bf S}, {\bf Rk1}, {\bf I} or  {\bf Nn}.

 Let us fix  a  basis $\textbf{e}=(e_i)_{i=1,..d}$ of $\RR^d$. There exists a constant $c_{\textbf{e}}$ such that  $\Vert A\Vert  \leq  c_{\textbf{e}}\sum_{i=1}^d \vert A e_i\vert $ for  any $A\in \mathcal M(d, \mathbb R)$; hence, for any $\bar v \in  {\bf P}(\RR^d)$, 
\[
C_1(\bar v) 
:=\sup_{n \geq 1} {\Vert A_{n, 1} \Vert \over 	\vert  A_{n, 1}\bar v\vert } 
\leq c_{\textbf{e}} \sum_{i=1}^d  \sup_{n \geq 1} {\vert A_{n, 1} \bar e_i\vert \over \vert A_{n, 1}\bar v\vert}
=c_{\textbf{e}}\sum_{i=1}^d  C_1(\bar e_i, \bar v)
\]
where, for any $\bar u, \bar v \in {\bf P}(\RR^d)$, we set $\displaystyle C_1(\bar u, \bar v): = \sup_{n \geq 1} {\vert A_{n, 1}\bar u\vert \over \vert A_{n, 1}\bar v\vert}$.
Then, since $\ln^+$ is non decreasing and $\ln^+(ab)\leq \ln^+(a)+\ln^+(b),$
$$ \ln^+C_1(\bar v)\leq \ln^+ c_{\textbf{e}}+  \sum_{i=1}^d \ln^+ C_1(\bar e_i, \bar v).$$

 Hence, the RNC property $A_2(\beta)$    holds as soon as $ \mathbb E((\ln^+ C_1(\bar u, \overline{V}_0))^\beta)<+\infty$  for any  $ u$ in a set of generators of $\mathbf{e}$ of $\RR^d$.

Now, the quantity $\displaystyle {\vert A_{n, 1}\bar u \vert \over \vert A_{n, 1} \bar v \vert}$ may be decomposed as  
$
\displaystyle
   \prod_{k=1}^n {\vert A_{k} \bar u_{k-1} \vert \over \vert A_{k} \overline{v}_{k-1} \vert },
$
with $\bar u_0=\bar u, \overline{v}_0= \bar v$ and $\bar u_{k}=A_k\cdot \bar u_{k-1}, \overline{v}_{k}= A_k\cdot \overline{v}_{k-1}$ for any $k \geq 1$. Hence
\begin{equation}\label{eq-logRNC-dec}
	\ln^+ C_1(\bar u, \bar v)\leq \sum_{k=1}^{+\infty} \ln^+\left({\vert A_k \bar u_{k-1} \vert \over  \vert A_k \overline{v}_{k-1}\vert }\right).
\end{equation}

We will see in the following proposition that this sum converges is an the appropriate $L^p$ spaces in   the different situations presented above. 

Let us  point out  why the reasons of such convergence are different in each of the situations explored here. In the case of similarities  {\bf S},  it holds $\displaystyle {\vert A_n \bar u\vert \over \vert A_n\bar v\vert}=1$ for all $\bar u, \bar v\in\Pd$ thus  all the terms of the series in the right hand side of (\ref{eq-logRNC-dec}) vanish.   In the cases {\bf Rk1}, {\bf I} and {\bf Nn},  the convergence is due to the fact that the distance between $\bar u_k$ and $\overline{V}_k$ goes to zero at exponential speed.

\begin{prop}\label{prop-RCN}  Let $\nu$ be a $\mu$ invariant measure supported by $\SS^{d-1}$ and $\overline{V}_0\sim \nu$ then, under hypotheses {\bf S}, {\bf Rk1}, {\bf I} or {\bf Nn},  the RNC property   $\textbf{A}_2(\beta)$ holds (i.e. $\mathbb E((\ln^+ C_{ 1}(\overline{V}_0))^\beta)<+\infty)$  for $\beta>3$ .
		
\end{prop}

\begin{proof} 
	$\bullet$ \underline{Case {\bf S}}.
	The case {\bf S} is straightforward since $C_1(v)=1$ for any $v \in {\bf P}(\RR^d)$  in this case. The other cases rely on the contraction property satisfied by the process $(\overline{V}_k)_{k \geq 0}$. Let us first propose a general criterion then apply it in these 3 remaining situations.

 $\bullet$ \underline{Case {\bf Rk1}}. In this case,  $\bar u_k=\overline{V}_k=\overline{w}_k$ for any $k\geq 1$ so that
	$$ 	\ln^+ C_1(\bar u,  \bar v)\leq \ln^+\left({\vert A_1\bar u\vert \over  \vert A_1 \bar v\vert }\right)\leq \ln^+\left({\vert \langle\tilde w_1,u\rangle\vert \over  \vert \langle\tilde w_1,v\rangle \vert }\right) )\leq \ln^+\left({ 1 \over  \vert \langle\tilde w_1,v\rangle \vert }\right).  $$
	Thus  $\textbf{A}_2(\beta)$ is a direct consequence of  hypothesis \textbf{Rk1}-\textit{Reverse moments} with  $3< \beta <3+\varepsilon$.
	
	To deal with cases \textbf{I} and \textbf{Nn}, we   use the following  lemma.

	\begin{lemma}\label{lem-cont-RCN} We assume that there exists an invariant mesure $\nu$  for the chain $(\overline{V}_k)_{k \geq 0}$ with support included in ${\bf P}(\RR^d)$,   constants $\beta >2, C>0, \rho \in (0, 1)$  
		
	(i) $\displaystyle{\mathbb E\left(\left(\ln^+\frac{\Vert A_1\Vert }{\vert A_1 \overline{V}_0 \vert } \right)^{b}
\right)}$ for some $b>2\beta $ where $\overline{V}_0$ has distribution $\nu$;

(ii) $\displaystyle \mathbb E\left(\delta(A_{k, 1}\cdot \bar u, A_{k, 1}\cdot \overline{V}_0) \right)\leq C \rho ^k$   for some fixed $\bar u\in \Pd$. 

Then $\quad \displaystyle \mathbb E\left((\ln^+ C_1(\bar u, \overline{V}_0))^{\beta}\right)< +\infty$.
	\end{lemma}
 Proof. First,   for any $\beta\geq 1$,   inequality  (\ref{eq-logRNC-dec}) yields 
	\[
	\mathbb E\left((\ln^+ C_1(\bar u, \overline{V}_0))^\beta\right)^{1/\beta}
	\leq  \sum_{k \geq 1} \mathbb E\left(\left(\ln^+{\vert A_k \bar u_{k-1}\vert \over  \vert A_k \overline{V}_{k-1}\vert }\right)^\beta\right)^{1/\beta} 	\] 
	Note that,  for any $A \in M(d, \mathbb R)$ and  any two directions $\bar u, \bar  v \in {\bf P}(\RR^d)$ with representative vectors $u, v\in \Sd$,  
	\begin{equation*}
		\ln^+ 	{\vert A\bar u\vert \over \vert A\bar v\vert}\leq \ln^+  {\Vert A\Vert \over \vert A\bar  v\vert}= \ln^+ {\mathfrak n} (A, \bar v)
	\end{equation*}
	with ${\mathfrak n} (A, \bar v):=  \Vert       A \Vert / \vert A\bar v\vert\in[1,+\infty]$.
Furthermore  
	\begin{equation*}  
	\ln^+ 	{\vert A\bar u\vert \over \vert A\bar v\vert}\leq \ln^+  {\vert Av\vert +\vert A(u-v)\vert \over \vert Av\vert}\leq  \ln^+ (  1+{\mathfrak n} (A, v)\vert u-v\vert) 	
	\end{equation*}
	and
	\begin{equation*}  
	\ln^+ 	{\vert A\bar u\vert \over \vert A\bar v\vert}\leq \ln^+  {\vert A(-v)\vert +\vert A(u+v)\vert \over \vert A\bar v\vert}\leq  \ln^+ (  1+{\mathfrak n} (A, \bar v)\vert u+v\vert) 	
	\end{equation*}
	so that 
	\begin{equation} \label{kazjcehbkc}
	\ln^+ 	{\vert A\bar u\vert \over \vert A\bar v\vert}\leq  \ln^+ (  1+{\mathfrak n} (A, \bar v)\min( \vert u-v\vert, \vert u+v\vert))\leq \sqrt{2} \ {\mathfrak n} (A, \bar v)  \ \delta(\bar u, \bar v). 	
	\end{equation}
In particular, for any $k \geq 1$,  setting $a_k=\rho^{-k/2\beta}$, 
	\begin{align*}
\left(	\ln^+ {\vert A_k\bar u_{k-1}\vert \over  \vert A_k\overline{V}_{k-1}\vert}\right)^\beta
&= 
\left(	\ln^+ {\vert A_k\bar u_{k-1}\vert \over  \vert A_k\overline{V}_{k-1}\vert}\right)^\beta  {\bf 1}_{[\mathfrak n(A_k, \overline{V}_{k-1})> a_k]}+ 
\left(	\ln^+ {\vert A_k\bar u_{k-1}\vert \over  \vert A_k\overline{V}_{k-1}\vert}\right)^\beta {\bf 1}_{[\mathfrak n(A_k, \overline{V}_{k-1})\leq a_k]} 
\\	\leq	& 
\left(	\ln^+ \mathfrak n(A_k, \overline{V}_{k-1}) \right)^\beta  {\bf 1}_{[\mathfrak n(A_k, \overline{V}_{k-1})> a_k]}+ 
\left( \sqrt{2} \ 	\mathfrak n(A_k, \overline{V}_{k-1}) \ \delta(\bar u_{k-1},\overline{V}_{k-1})\right)^\beta {\bf 1}_{[\mathfrak n(A_k, \overline{V}_{k-1})\leq a_k]} 
\\	\leq	& 
\left(	\ln^+ \mathfrak n(A_k, \overline{V}_{k-1}) \right)^\beta {\left(	\ln^+ \mathfrak n(A_k, \overline{V}_{k-1}) \right)^{\beta' }\over (\ln^+ a_k)^{\beta' }}  + 2^{\beta/2} \  a_k^\beta \ \delta(\bar u_{k-1},\overline{V}_{k-1})^\beta
\\	\leq	& \left(\frac{2\beta}{-(\ln\rho)}\right)^{\beta' }
 { \left(	\ln^+ \mathfrak n(A_k, \overline{V}_{k-1}) \right)^{\beta + \beta' }\over k^{\beta' }}  + 2^{\beta/2} \rho^{-k/2} \delta(\bar u_{k-1},\overline{V}_{k-1})
	\end{align*}
 since $\delta(\bar u_{k-1},\overline{V}_{k-1})\leq 1$ and $\beta>1$. 
Consequently, there exist   positive  constants $C(\beta,\beta',\rho)$ and $C'$ such that
\begin{align*}
	\mathbb E\left(\left(\ln^+{\vert A_k\bar u_{k-1}\vert \over  \vert A_k \overline{V}_{k-1} \vert }\right)^\beta\right)^{1/\beta} &\leq C(\beta,\beta',\rho)\ 
	 \left( \frac{\mathbb E\left(\ln ^+\mathfrak n(A_1, \overline{V}_{0})^{\beta + \beta' }\right)}{k^{\beta' }}+  \rho^{-k/2} \EE(\delta(\bar u_{k-1},\overline{V}_{k-1}))\right)^{1/\beta } 
	 \\&\leq  C(\beta,\beta',\rho)\  
	 \left( \frac{\mathbb E\left(\ln^+\mathfrak n(A_1, \overline{V}_{0})^{\beta + \beta' }\right)}{k^{\beta' }}+  C'\  \rho^{k/2}  \right)^{1/\beta }. 
\end{align*}
Since  $\beta'=b-\beta>\beta$,  the   sequence 
$\displaystyle 
\left(\mathbb E\left(\left(\ln^+{\vert A_k\bar u_{k-1}\vert \over  \vert A_k \overline{V}_{k-1} \vert }\right)^\beta\right)^{1/\beta} \right)_{k \geq 1}
$
is summable, hence  $\ln^+ C_1(\overline{V}_0)$ has   a moment of order $\beta$.


\rightline{$\Box$}

Let us now apply this lemma.

$\bullet$ \underline{Case {\bf I}}.
Under hypotheses {\bf I}, the assumptions of Lemma \ref{lem-cont-RCN} are satisfied for any $\bar u\in\Pd$.  Indeed by \cite{BL}, proximality, strong irreductibility  and  exponential moments  hypotheses imply condition \textit{(ii)}  of Lemma \ref{lem-cont-RCN}  (i.e uniform exponential contraction in mean) for every $\bar u$ and $\bar v\in \Pd$.  

As a matter of  fact, the existence of exponential moments  allows a more direct proof, since it  yields to the following inequality: there exist constants $C>0$ and $ \rho \in [0, 1)$ s.t.
$$\mathbb E\left(\left(\ln^+ C_1(\bar u, \overline{V}_0)\right)^p\right) \leq C  \ 
	\mathbb E(\mathfrak n(A_1, \overline{V}_0)^{\epsilon })
	\ \sum_{n=1}^{+\infty}\rho^{n} < +\infty.
	$$
 As  mentioned in section \ref{sect-I},  some work  in progress \cite{Pen23} seems able  to relax this strong moment assumption, ensuring that condition \textit{(ii)}  of Lemma \ref{lem-cont-RCN} holds without moment condition. If this is confirmed Lemma \ref{lem-cont-RCN} can be applied only under polynomial moments.   

$\bullet$ \underline{Case {\bf Nn}}
In this case, under hypotheses {\bf Nn}, the support of $\nu$ is included in $\mathbb X$.  By (\ref{meanN}) we have that for all $u, v\in \XX,$  
$$\mathbb E\left(\delta(A_{k, 1}\cdot \bar u, A_{k, 1}\cdot \bar v) \right)\leq 2\  \mathbb E\left(\mathfrak d(A_{k, 1}\cdot u, A_{k, 1}\cdot  v) \right)\leq 2 \ c\ \rho ^k\  \mathfrak{d} (u, v)\leq 2\ c\  \rho ^k$$
Hence, the assumptions of Lemma \ref{lem-cont-RCN} are satisfied for any $u\in \XX$,  taking $b>6$ in order to get $\beta >3$.	
 	
 \end{proof}

 \noindent {\bf Proof of Theorem \ref{thm-main}} 
Theorem \ref{thm-main} is a direct consequence of criterion \ref{thm-from-RCN-to-cons}. Indeed, in the four cases explored here:

$\bullet$ Assumption  ${\bf A}_1(\alpha)$ is satisfied for $\alpha \in (0, 1/2)$, by Proposition \ref{zekbghjkd};  

$\bullet$ In the case  {\bf S}, assumption {\bf B}$(\beta)$ is satisfied for any $\beta>0$ since the  RNC random variable   vanishes in this case. 

In the  case {\bf Rk1},  this assumption holds for some $\beta >0$  by hypothesis {\bf Rk1}-{\it Reverse moments}. 

In the two remaining cases, this is a consequence of Proposition \ref{prop-RCN}.

$\bullet$ The last assumption {\bf B}($\gamma)$ does not depend on the nature of the matrices $A_k$.

We end this section by discussing the  irreducibility assumptions of the linear part.  
\begin{remark} : {\bf optimality of the irreducibility assumptions}

Here we give two examples of transient affine recursions on $\RR^2$ in the critical case where  the irreducibility assumptions on the linear part $\overline{\mu}$ fail. 

\underline{Example 1}(reducible linear part) 

It is clear that Theorem 1.1 fails if we only assume that $\supp(\overline{\mu})$ is reducible as one can see by taking $A_1=I_2$ almost surely. Here is another less trivial example. 
 Let     $\mu$ be a probability measure on $\textrm{Aff}(\RR^2)$ such that $\overline{\mu}$ is supported on the set of  $2\times 2$-diagonal matrices $A=\begin{bmatrix}\alpha(A)& 0 \\
 0 & \alpha^{-1}(A)\end{bmatrix}$ of determinant one. This is a reducible model as each of the coordinate axis is invariant under the support of $\overline{\mu}$. 
 Assume that  $\int{\log \alpha(A) d\overline{\mu}(A)}=0$,  $\overline{\mu}(\{A : \alpha(A)=1\})<1$ and  $\mu(\{g : g\mathcal A= \mathcal A\})<1$  for every 
 proper non-empty affine subspace $\mathcal A$ of $\RR^2$. Assume also that $\mu$ has a moment of order $2+\delta$ for some $\delta>0$.   Then the process $(X_n)_{n\in \NN}$ is transient. 
 Indeed, setting 
$X_{n}^0:=(x_n,y_n)$, we notice that the diagonal structure of $A_n\cdots A_1$ implies that  
 both $x_{n}$ and $y_{n}$ are critical one-dimensional affine recursions  with respective linear part $\alpha(A_n\cdots A_1)$ and $\alpha^{-1}(A_n\cdots A_1)$. Each of these affine recursions satisfies  assumptions (H) of Babillot--Bougerol--Elie's local contraction's theorem \cite[Theorem 3.1]{BBE}. The aforementioned theorem yields    that for any $K>0$, almost surely, $\alpha(A_n\cdots A_1){\bf 1}_{|x_n|\leq K}\underset{n\to +\infty}{\to} 0$ and $\alpha^{-1}(A_n\cdots A_1){\bf 1}_{|y_n|\leq K}\underset{n\to +\infty}{\to} 0$. 
It follows that,  almost surely, $|X_n^0|\underset{n\to +\infty}{\to} +\infty$, hence  the process $(X_n)_{n\in \NN}$ is   topologically transient.

\underline{Example 2} (irreducible but not strongly irreducible linear part) 

Consider a probability measure $\mu$  on $\textrm{Aff}(\RR^2)$  whose  projection $\overline{\mu}$ on  $\textrm{GL}_2(\RR)$ satisfies  $\textrm{Supp}(\overline{\mu})=\left\{\begin{bmatrix}0 & \lambda\\
    \lambda^{-1}& 0 \end{bmatrix}, \begin{bmatrix}0& 1\\
    1 & 0 \end{bmatrix}\right\}$ with $\lambda>1$. It is immediate that  the support of $\overline{\mu}$ is irreducible  but not strongly (as it permutes the two coordinate axis). Assume also that there is no proper affine subspace  of $\RR^2$ fixed by all elements of the support of $\mu$. 
   The process $(X_{2n})_{n\in \NN}$ is the affine recursion of $\RR^2$ induced by the probability measure $\mu^{\star 2}$ on $\textrm{Aff}(\RR^2)$.
Observe that $\overline{\mu}^{\ast 2}$ is supported on $\{I_2, \textrm{diag}(\lambda, \lambda^{-1}), \textrm{diag}(\lambda^{-1}, \lambda)\}$ and gives equal mass to the two diagonal non identity matrices.  It follows that  for every $n\in \NN$, $A_{2n}\cdots A_1=\begin{bmatrix}\lambda^{S_n}& 0 \\
0 & \lambda^{-S_n}\end{bmatrix}$ where $(S_n)_n$ is a centred random walk on $\ZZ$. Hence, from the strong law of large numbers,   $\gamma_{\bar \mu} =0$. It is easy to check that  there is no proper affine subspace of $\RR^2$ fixed by all elements of $\mu^{\ast 2}$. 

\noindent By  Example 1.~above, we deduce that 
the process $(X_{2n})_{n\in \NN}$ is  transient. Since $X_{2n+1}=A_{2n+1} X_{2n}+B_{2n+1}$ for every $n\in \NN$, and since   the support of $\mu$ is compact, we deduce that $|X_n|\underset{n\to +\infty}{\to} +\infty \  \PP$-almost surely. Hence,  the process $(X_n)_{n\in \NN}$ is  topologically transient.

\end{remark}

 \section{On  the fluctuations of $(\ln |A_{n, 1}v|)_{n \geq 1 } $ in the rank one case}\label{sectionfluctuationsRk1}

The process   $(\ln |A_{n, 1}v|)_{n \geq 1 } $ fits into the general context of Markov walks, that is random walks whose increments are governed by   an underlying Markov chain $(Z_n)_{n \geq 0}$. The fluctuations of Markov walks have been studied in  \cite{GLL}, Theorem 2.2, under some general hypotheses on the  chain $(Z_n)_{n \geq 0}$. There are several ways to apply  this result, we present here one suitably adapted to the situation.

 We denote here  by $ \Rpm$ the quotient space $\RR^d/\sim$ where $z \sim z'$ means $z =  z'$ or $z = - z'$, for any $z, z'$ in $\mathbb R^d$;   the class of $z$  equals $\pm z:=\{z, -z\}$. 
 The one-to-one correspondence 
\begin{align*}   \mathbb R\times \Pd  \quad &\to \quad  \Rpm\setminus \{0\}   
\\
 (\lambda,\bar v) \quad &\mapsto \quad \pm e^\lambda v  
\end{align*}
 (with $v \in \Sd$ a representative of $\bar v$) allows to identify $\RR\times \Pd$ with $ \Rpm\setminus \{0\}$ . In order to simplify the notation, the element $\pm z\in\RR^d_\pm$ is also denoted $z$ and is identified with  $(\ln\vert z\vert , \bar z)$ when $z\neq0$.

 Let  $(Z_n)_{n \geq 0}$ be the Markov chain on  $ \Rpm$ defined recursively by
$$ Z_n:=\pm A_nZ_{n-1}/\vert Z_{n-1}\vert  \mbox{ if }Z_{n-1}\neq 0\mbox{ and } Z_{n}=0  \mbox{ if }Z_{n-1}= 0$$
and whose transition kernel $P$ is  defined by:  for any Borel function $\varphi:  \Rpm  \to \mathbb R^+$ and any $z\in  \Rpm$,  
$$
  P \varphi(z):=\left\{
  \begin{array}{cll}
 \displaystyle  \int_{\mathcal M(d, \mathbb R)}\varphi\left(\pm A\frac{z}{\vert z\vert }\right)\  \bar \mu  ({\rm d} A)  &{\rm if}    &z\neq 0,
  \\
 \varphi(0)& {\rm if} & z=0.
\end{array}
\right.
$$
Lets us emphasize that $\overline{Z_n}=\overline{V_n}$ for any $n\geq 0$, where $(\overline{V}_n)_{n\geq 0}$ is the Markov chain  on the projective space defined in the previous sections. When $Z_n \neq 0$,     
$$Z_n\cong (\ln\vert Z_n\vert , \overline{Z_n} )=(\ln\vert A_n \overline{Z}_{n-1}\vert , \overline{Z_n}).$$ 
For any probability measure $\nu$ on  $\Pd$, let us denote  $\bar\mu *\nu $ the probability measure    on $ \Rpm$ defined by : for any Borel function $\varphi:  \Rpm  \to \mathbb R^+$,   
 $$
\bar\mu *\nu (\varphi):=\int \varphi(\pm Av) \bar{\mu}(dA)\nu(d\bar v),
.$$  When $\nu$ is $\bar \mu$-invariant, the measure $\bar \mu * \nu$ is a stationary probability measure for $(Z_n)_{n \geq 0}$  with support in  $\RR^d_\pm\setminus{0}$.

For $z \in \Rpm\setminus\{0\} $, set $f(z):=\ln \vert z\vert $. When $Z_0\neq 0$, the quantity $\ln  \vert A_{n, 1} Z_0\vert, n \geq 0,$ may be decomposed as 
$$
S_n(\overline{Z}_0)=\ln  \vert A_{n, 1} \overline{Z}_0\vert= \sum_{k=1}^{n} \ln \left\vert A_k \overline{Z}_{k-1}  \right\vert=\sum _{k=1}^{n} f({ Z}_k).
$$
The process $S_n(\overline{Z}_0)$ is a Markov walk on $\RR$ driven by the Markov chain $(Z_n)_n$ on $\RR^d_\pm$.

Under hypotheses {\bf Rk1},  the distribution of $\overline w_0$ is the unique  invariant measure $\nu$  for the process $(\overline Z_n)_{n \geq 0}$ on $\Pd$.
We assume that $  Z_0$ has distribution $\bar \mu *\nu$.

\begin{prop}\label{prop-fluct-Rk1}
	 Under hypotheses {\bf Rk1},  we have  $\EE(\ell_\rho(\overline{Z}_0)^\alpha)< +\infty$ for $\alpha<1/2$ . \end{prop}
	 
 \begin{proof}
Let $\varepsilon >0$ given by hypothesis {\bf Rk1}-{\it Moments} and 
	fix $\delta>1$ and $ p >2$  such that $\delta p <2+\varepsilon$. For all 
	$z\in \Rpm\setminus\{0\} $,  set 
	 
	$$
	N(z):=\sup_{n\geq 0}\left(P^n\vert f\vert ^{\delta p }(z)\right)^{1/ p }=\sup_{n\geq 0}\EE_z\left(\vert \ln\vert Z_n\vert \vert ^{\delta p }\right)^{1/ p }\in [0,\infty].
	$$
	Observe that,   under hypothesis {\bf Rk1}, for any  non negative  Borel   function $\varphi: \Rpm\to \mathbb R_+ $, the sequence  $(P^n\varphi(z))_{n \geq 0}$  converges ``very quickly'' towards  $\mub*\nu(\varphi).
	$
	Indeed,  $\overline{Z}_{n}=\overline  w_{n}$ for any $n \geq 0$ ; thus $\overline{Z}_n$ is independent from the starting point  in $\mathcal{X}_0:=\left\{ z\in \Rpm|\  \PP(|A_1z|=0)=0 \right\}$. This yields, for any  $z\in\mathcal{X}_0$ and $n\geq 2$,
	\begin{equation}\label{eq-sg-RK1}
		P^n\varphi(z)=\EE_z(\varphi(\pm A_n Z_{n-1}/\vert Z_{n-1}\vert))=\EE(\varphi(\pm A_1 Z_0/\vert Z_0\vert))=\mub*\nu(\varphi).
	\end{equation}
 	We claim that this ensures that for all $n\in \NN$ and $z\in\Rpm$, 
	\begin{eqnarray}\label{eq-PN}
		P^nN(z)\leq 3 N(z)\mbox{ and } (P^nN^ p (z))^{1/ p }\leq 3 N(z).
	\end{eqnarray}
	Indeed, $N(z)=+\infty$ when $z\not\in\mathcal{X}_0$ and the inequalities are trivially satisfied; when   $z\in\mathcal{X}_0$, the identity    (\ref{eq-sg-RK1}) implies  that 
	$N(z)\leq \vert f\vert ^\delta(z) + \left(P\vert f\vert ^{\delta p }(z)\right)^{1/ p }+(\mub*\nu(\vert f\vert ^{\delta p }))^{1/ p }$.
	Then, by Jensen inequality, since $1/ p <1$,
	\begin{eqnarray*}
		P^nN(z)&\leq& P^n\vert f\vert ^\delta(z)+P^n\left(P\vert f\vert ^{\delta p }\right)^{1/ p }(z)+ (\mub*\nu(\vert f\vert ^{\delta p }))^{1/ p }\\
		&\leq& \left(P^n\vert f\vert ^{\delta p }(z)\right)^{1/ p }+\left(P^{n+1}\vert f\vert ^{\delta p }(z)\right)^{1/ p } +(\mub*\nu((\vert f\vert ^{\delta p }))^{1/ p }\leq 3N(z).
	\end{eqnarray*}
	Similarly, by the $L_p$-triangle  inequality,
	\begin{eqnarray*}
		\left(P^nN^ p (z)\right)^{1/ p  }&\leq& \left(P^n\vert f\vert ^{\delta p }\right)^{1/ p }(z)+
		\left(P^n\left(P\vert f\vert ^{\delta p }\right)^{\frac{1}{ p } p }(z)\right)^{1/ p  }+(\mub*\nu(\vert f\vert ^{\delta p }))^{1/ p }\\
		&\leq& \left(P^n\vert f\vert ^{\delta p }(z)\right)^{1/ p }+\left(P^{n+1}\vert f\vert ^{\delta p }(z)\right)^{1/ p } +(\mub*\nu(\vert f\vert ^{\delta p }))^{1/ p }\\
		&\leq &3N(z).
	\end{eqnarray*}
	Consequently, it is of interest to consider the subset $\XN$ of $\Rpm$ defined by
$$\XN:=\left\{z\in\Rpm \mid   N(z)< +\infty \right\}.$$
 Under  hypothesis {\bf Rk1}, the set  $\XN$ has full $\bar\mu *\nu$-measure  since 
 $$\EE_{\bar\mu *\nu}(N(Z_0))\leq 3 \EE(\vert \log\vert A_1\overline{V}_0\vert \vert ^{\delta p })^{1/ p }< +\infty$$
  in this case. Thus we can restrict the Markov operator $P$ to spaces of functions restricted to   $\XN$. Let us consider the  space $\mathcal B$ of Borel  functions $\mathbb C$ defined by
	\begin{equation} \label{zkeg fbej}  
		\mathcal B= \left\{\varphi: \XN \to \mathbb C\mid \underbrace{\sup_{z\in \XN } \frac{ \vert \varphi(z)\vert}{1+N(z)}}_{=:\vert \varphi \vert _{\mathcal B} } <+\infty \right\}
	\end{equation} 
	The space $(\mathcal B, \vert \cdot \vert_{\mathcal B})$ is a Banach space over $\mathbb C$.

	We now restate   hypotheses M1-M5 introduced in \cite{GLL} and check that  under hypotheses {\bf Rk1} they  are  satisfied  on  the   Banach space $\mathcal B $ by the  kernel $P$ and the function  $f:  \Rpm\setminus\{0\} \to \mathbb R, z\mapsto f(z):=\ln \vert z\vert $.   This allows us to  apply Theorem 2.2 in \cite{GLL} and  conclude  the proof of   Proposition  \ref{zekbghjkd}.

	\underline {Hypothesis M1}
	
	\textit{M1.1. The constant function $1$ belongs to $\mathcal B$.}
	
	This is a trivial consequence of the fact that $1\leq 1+N$.
	
	\textit{M1.2. For any $z \in \XN $,  the Dirac mass $\delta_z$ at $z$ belongs to   
	  the topological dual space $\mathcal \cB' $ of $\cB$,  endowed with the norm $\displaystyle  \vert  \Phi\vert _{\mathcal B'}=\sup_{\stackrel{\varphi \in \mathcal B}{\varphi \neq 0}}{\vert \Phi(\varphi)\vert \over \vert \varphi\vert_{\mathcal B} }.$}
	
	Indeed, by definition of the norm on $\cB$,  for any $z\in \XN $,
	\begin{align*}
	\vert  \delta_{z}\vert _{\mathcal B'}&=\sup_{\stackrel{\varphi \in \mathcal B}{\varphi \neq 0}}{\vert \varphi(z)\vert \over \vert \varphi\vert_{\mathcal B}} \leq   1+N(z)<+\infty.
	\end{align*}

	\textit{M1.3. The Banach space $\mathcal B$ is included in $\mathbb L^1 (P(z, \cdot ))$, for any $z \in \XN $.}
	
	For any  $\varphi\in\cB$ and  $z\in \XN $,
	\begin{align*}
	\int \vert \varphi(z_1)\vert  P((z,dz_1)& 
	\leq\vert \varphi\vert _\cB \int (1+N(A_1z/\vert z\vert )) \bar \mu  (dA_1)
	\\
	&=\vert \varphi\vert _\cB(1+PN(z))
	 \leq \vert \varphi\vert _\cB(1+3N(z))<+\infty.
	\end{align*}

	\textit{M1.4. There exists $\epsilon_0>0$  such that,  for any $\varphi \in \mathcal B$,   the function $e^{it f}\varphi$ belongs to $\mathcal B$ for every $t \in ]-\epsilon_0, \epsilon_0[$.}
	
	This is a direct consequence of the equality $\vert e^{it f}\varphi\vert =\vert \varphi\vert $ valid for any $t \in \mathbb R$. Property \textit{M1.4. } thus holds for any $\epsilon_0>0$.

	\underline {Hypothesis M2}
	
	\textit{{M2.1. The map $\varphi \mapsto P\varphi$ is a bounded operator on $\mathcal B$.}}\\	
	Indeed,  for any $\varphi \in \cB$ and $z \in \XN$,
	$$\vert P\varphi(z)\vert \leq \EE_z(\vert \varphi(Z_1)\vert )\leq \vert \phi\vert _\cB\EE_z(1+N(Z_1))\leq 3 \vert \varphi\vert _\cB  (1+N (z)) $$
	which implies $P\varphi \in \cB$ and  $\vert P\vert _\cB\leq 3 $.  
	
	\textit{M2.2. There exist constants $C_P>0$  and $\kappa\in [0, 1)$ such that $P = \Pi + Q$,
		where $\Pi$ is a one-dimensional projector and $Q$ a bounded  operator on $\mathcal B$ with spectral radius $<1$ and  satisfying $\Pi Q = Q\Pi = 0$.		
	}\\
	For any $\varphi \in \cB$, set $\Pi(\varphi):=\mub*\nu(\varphi)$.  The equality $ \Pi P=P \Pi =\Pi$ holds since
  $\mub*\nu$ is $P$-invariant.
	
	Thus, with    $Q:=P -\Pi$, it holds 
	$\Pi Q = Q\Pi = 0$ and $P^n=\Pi +Q^n$.
	Furthermore,  equality (\ref{eq-sg-RK1}) may be rewritten as  $P^n=\Pi$,  
	thus $Q^n=0$,  for $n\geq 2$.

	\underline {Hypothesis M3}

	{\it  
		For any $t  \in ]-\epsilon_0, \epsilon_0[$, any $\varphi \in \cB$ and $z \in \XN$,   let $P_t \varphi(z):=\EE_z(e^{itf(Z_1)}\varphi(Z_1))$. Then,  the map $\varphi \mapsto P_t \varphi$ is a bounded operator on $\mathcal B$  such that 
		\[
		\sup_{\stackrel{n \geq 1}{\vert t\vert \leq \epsilon_0}}\Vert P^n_t\Vert_{\mathcal B\to \mathcal B}<+\infty.
		\]
	}
	
	This hypothesis is satisfied in a obvious way here   since $\vert e^{itf} \varphi\vert_\mathcal B=\vert   \varphi\vert_\mathcal B$. Hence $\Vert P^n_t\Vert_{\mathcal B\to \mathcal B}= \Vert P^n\Vert_{\mathcal B\to \mathcal B}\leq 3 $ for any $n \geq 0$ and $t \in \mathbb R$.
	
	\underline {Hypothesis M4} (Local integrability). 
	{\it There exist   bounded  non negative functions $N_0, N_1, N_2, \ldots$  in $\mathcal B$  and  constants $c,\beta,  \gamma>0$  and $ p_0 >2$ such that,  for any 
		$z\in \XN $ and $k \geq 1$, 
		
		1. 
		$\displaystyle 
		\max \left\{
		\vert f( z )\vert ^{1+\gamma}, \Vert \delta_z\Vert _{\mathcal B' },
		\sup_{n \geq 1} \left(\mathbb E_ z ^{1/ p_0 }(N_0(Z_n)^{ p_0 })\right)\right\}
		\leq c (1+N_0( z )); 
		$

		2. $N_0( z ) {\bf 1}_{\{N_0( z )\geq k\}}\leq N_{k}( z )$;
		
		3.  $\displaystyle \vert\Pi (N_k)\vert\leq {c\over k^{1+\beta}}$.
	}
	
	Recall that $\delta >1$ and $p>2$ are fixed and satisfy  $\delta p < 2+\varepsilon$ where $\varepsilon$ is given by hypothesis {\bf Rk1}-{\it Moments}. We set  $N_0:=N$, $N_k:=N( z ) {\bf 1}_{\{N( z )\geq k\}}, k \geq 1,$  and $p_0=p$, $c=3, \beta = p-2, \gamma=\delta-1.$ With these functions and constants, 	hypothesis   { \it M4} is satisfied since : 
	
	{\it M4.1})
$\vert f(z)\vert ^{1+\gamma}=(P^0\vert f\vert ^{\delta p })^{1/ p }\leq N(z)\leq c (1+N(z)), $
	$\displaystyle  \vert  \delta_{z}\vert _{\mathcal B'} \leq 1+N(z)$ by {\it M1.2} and 
	$\mathbb E_ z ^{1/ p }(N(Z_n)^{ p })=\left(P^nN^ p (z)\right)^{1/ p  }\leq 3N(z)$ for any $n \geq 1$, by (\ref{eq-PN});

	{\it M4.2}) This condition is satisfied since $N_k:=N( z ) {\bf 1}_{\{N( z )\geq k\}}, k \geq 1$;

 {\it M4.3}) For any $k \geq 1$, 
	\begin{eqnarray*}
		\Pi(N_k)& = & \EE_{\mu *\nu}(N(Z_0) {\bf 1}_{\{N(Z_0)\geq k\}})
		\\
		&\leq& \EE_{\mu *\nu}\left(N(Z_0) \frac{N( Z_0)^{ p -1}}{k^{ p -1}}\right)\\
		&\leq& {\EE_{\mu *\nu}(N(Z_0)^{ p })\over k^{ p -1}} \leq 3{ \EE_{\mu *\nu}(\vert \ln\vert Z_0\vert ^{  \delta p}) \over k^{1+( p -2)}} \quad {\rm with}  \quad \EE_{\mu *\nu}(\vert \ln\vert Z_0\vert ^{  \delta p}) <+\infty. 
	\end{eqnarray*}
	 
	\underline {Hypothesis M5} . {\it We suppose that 
		$\displaystyle \gamma_{ \bar \mu  }= \lim_{n \to +\infty}{1\over n} \mathbb E_{z}(f(Z_n))=0$ 
		and  }
	$$\sigma^2:=  \lim_{n \to +\infty}{1\over n} \mathbb \overline{V}_{z}\left(\sum_{k=1}^nf(Z_k)\right)>0.$$
	Equality  (\ref{eq-Sn_Rk1}) and hypothesis {\bf Rk1}-{\it Centring}  yield  $\gamma_{\bar \mu}= \mathbb E(\ln a_1)+\mathbb E(\ln \vert \langle \tilde w_2, w_1 \rangle\vert)=0$. 
	
	\noindent Let us now compute the asymptotic variance of $S_n$. Since  $f(Z_k)$ and $f(Z_\ell)$  are independent when $\ell \geq k+2$, it holds
	 \begin{align*}
		{1\over n}\mathbb \overline{V}_{z}\left(\sum_{k=1}^nf(Z_k)\right)&= {1\over n}\sum_{k=1}^n
		 V_z(f(Z_k)) +{2\over n}  \sum_{k=1}^{n-1} {\rm cov}_z(f(Z_k), f(Z_{k+1}))\\
		&{\stackrel{\tiny n \to +\infty}{\longrightarrow}} \quad V (f(Z_2))+2\  {\rm cov}(f(Z_2), f(Z_{3})).
	\end{align*}
    Notice that the asymptotic variance does not depend on $z$ since  the random variables $(f(Z_k), f(Z_{k+1})$  are independent of $Z_0=z$  and have the same distribution as $(f(Z_2), f(Z_{3})$ when $k \geq 2$; furthermore, it is non zero by {\bf Rk1}-{\it Non degeneracy} assumption.

	Finally, under hypotheses {\bf Rk1}, the chain $(Z_n)_{n \geq 0 }$ and the function $f$ satisfy conditions M1-M5 of   \cite{GLL} ; hence,  Theorem 2.4.2 in  \cite{GLL} is valid in our context, which readily implies that there exists a constant $c_\rho >0$ such that, for any $z \in \XN$ with modulus 1,   
	$$
	\PP(\ell_1^{(\rho)}( \bar z)>n)\leq  { c_\rho\over \sqrt{n}}(1+N(z)). 
	$$
In particular, if $Z_0$  has distribution $\bar \mu * \nu$, modulus 1 and $\overline{Z}_0=\overline{V}_0$ has distribution $\nu$, we obtain,   
$$\mathbb P\left( \ell_1^{(\rho)}(\overline{V}_0) >n\right)\leq  {c_\rho\over \sqrt{n}} \overline{\mu}*\nu(1+N).
$$   This imply that condition ${\bf A}_1( \alpha )$ holds for any $ \alpha  <{1\over 2}$.

\end{proof}

  \section{Appendix: generalization to non-proximal invertible matrices}\label{appendix}
When the matrices $A_n$ are invertible, it is possible to generalize the previous results in  the case where  hypothesis {\bf I}-{\it Proximality} fails.
We   only exclude the degenerate case where $T_{\bar\mu}$ is    conjugate to a subsemigroup of the group $\mathbb R^* O_d(\mathbb R)$ of Euclidean similarities, corresponding, after a change of 
Euclidean structure,  to a scalar cocycle already covered by case {\bf S}.  By \cite[Theorem 4.11]{BQ}, the remaining cases satisfy a non-degenerate central limit theorem.
\begin{theo} \label{exterior} Assume that  the matrices $A_n, n \geq 1$, are invertible and satisfy hypotheses  

\centerline{{\bf I}-{\it (Exponential) Moments and Reverse moments},\quad {\bf I}-{\it Strong irreducibility}\quad  and \quad {\bf I}-{\it Centring}.}
Assume moreover that $T_{\overline{\mu}}$ is not conjugate to a subsemigroup of 
$\mathbb{R}^* O_d(\mathbb{R})$. 
Then,  if  there exists $\gamma >2$ such that  $\mathbb E\left((\ln^+\vert B_1\vert)^{\gamma}\right)<+\infty$, the Markov chain $(X_n^x)_{n \geq 0}$ is  recurrent.
\end{theo} 
\begin{proof}
	 Let $T_\mu$ be the  sub-semigroup of $Gl(d, \RR)$ generated  by the support of $\mu$. Let $r=r(T_\mu)\geq 1$ be  the {\it proximal dimension} of $T_\mu$ in $\mathbb R^d$ ,  that is the least rank of a nonzero element  $M$ of the closure
	 $$
	 \overline{\mathbb R T_\mu}:= \{ M \in \mathcal M(d, \mathbb R)\mid M=\lim_{n \to +\infty} \lambda_n M_n \ with \ \lambda_n \in \mathbb R, M_n \in T_\mu\}. 
	 $$   	
	Let us now consider the action of ${\rm Gl}(d,  \mathbb R)$ on the  $r$-exterior product $\wedge^r\RR^d$ and denote by $\|\wedge^r A\|$ the related operator norm of $A\in {\rm Gl}(d,  \mathbb R)$.   
	 By \cite[Lemma 4.13]{BQ}, there exists a $T_\mu$-invariant subspace ${\cW_r}$ of $\wedge^r \RR^d$   such that 
	 the action of $T_\mu$ on   ${\cW_r}$ is proximal (i.e. $T_\mu$ contains a proximal element)  
	 and strongly irreducible. Moreover, there exists $C\geq 1$  such that, for any $A\in T_\mu$, one has 
	 \begin{equation}\label{eq-norm-wedge}
	 	C^{-1}\|A\|^r\leq\|\wedge^rA\|_{_{\cW_r}}\leq \|A\|^r 
	 \end{equation}
	 (where  $\|\wedge^rA\|_{_{\cW_r}}$ denotes  the norm  of the restriction of $\wedge^rA$  to ${\cW_r}$).

 We claim that the action of the random matrices $A_n$ on ${\cW_r}$ by the standard representation $
 {\rm Gl}(d, \mathbb R) \to  {\rm Gl}(d, \wedge^r\mathbb R)$ satisfies   hypotheses \textbf{I}.  Indeed, by  (\ref{eq-norm-wedge}),    
 
 \indent $(i)$ Assumptions {\bf I}-{\it (Exponential) Moments and Reverse moments} on $\RR^d$,  with exponent $\varepsilon>0,$ imply :
 $$ \EE\left( \Vert \wedge^rA_1\Vert_{_{\cW_r}} ^{\varepsilon/r}\ \right)\leq \EE\left( \Vert A_1\Vert ^{r\varepsilon/r}\ \right) <+\infty\mbox{ and } \EE\left( \Vert \wedge^rA_1^{-1}\Vert_{_{\cW_r}} ^{\varepsilon/r}\ \right)\leq  \EE\left( \Vert A_1^{-1}\Vert ^{r\varepsilon/r}\ \right) <+\infty. $$
 In other words,  the same  hypotheses hold  for the standard representation on ${\cW_r}$,  with exponent $\varepsilon/r$.
 
$(ii)$
 The  Lyapunov exponent $\gamma_{\bar\mu}^{^{\cW_r}}$ of $\wedge^r A_n$ on ${\cW_r}$ equals $ r \gamma_{\bar\mu}$, since by (\ref{eq-norm-wedge}) 
 $$r \frac{\EE(\ln\|A_{n,1}\|)}{n} -\frac{\ln C}{n}\leq \frac{\EE(\ln\|\wedge^rA_{n,1}\|_{_{\cW_r}})}{n} \leq r \frac{\EE(\ln\|A_{n,1}\|)}{n}.$$  
In particular the centring hypothesis holds for the action of the $A_n$ on $\RR^d$ if and only if it holds for the action of  $\wedge^r A_n$ on ${\cW_r}$.
	 
 Consequently assumptions  {\bf I} hold on ${\cW_r}$ and there exists a unique stationary probability measure $\nu_{_{\cW_r}}$  on ${\bf P}({\cW_r})$.  Let $\overline{W}_0$ be a random variable with distribution $\nu_{_{\cW_r}}$.
 
 As in section \ref{sectionrecurrence}, we consider  the ladder times $\ell_1^{^{(\rho), {\cW_r}}}, 0< \rho <1$ and  the reverse norm coefficients $ C_k^{^{\cW_r}} $ associated to  the standard representation on ${\cW_r}$  as follows:  for any $\overline{w} \in \mathbb P_1({\cW_r}),$	 
	$$ 
\ell_1^{^{^{(\rho), {\cW_r}}}}(\overline{w}):= \inf \{n \geq 1 \mid  \vert A_n \cdots A_1\overline{w}\vert_{_{\cW_r}} \leq \rho\} \mbox{ and }
C_k^{^{\cW_r}}(\overline{w}):= \sup_{n \geq k} {\Vert\wedge^r A_n\cdots A_{k}\Vert_{_{\cW_r}} \over 
		\vert A_n\cdots A_{k} \overline{w}\vert_{_{\cW_r}} }.$$	
These ladder times are almost surely finite since, by Theorem~4.11 of \cite{BQ} together with \eqref{eq-norm-wedge}, the log-norm cocycle on $\mathcal W_r$ is centered with positive variance. 
In order to control the norms    $\Vert A_n\ldots A_k\Vert $ from the random variables  $ C_k^{^{\cW_r}}(\overline{w})$,  we   introduce a ``generalized reverse norm control coefficient''  defined as follows: 
		$$C_{ k}^{^{\RR^d,{\cW_r}}}(\overline{w}):= \sup_{n \geq k} {\Vert A_n\cdots A_{k}\Vert \over 
	\vert A_n\cdots A_{k} \overline{w}\vert_{_{\cW_r}}^{1/r} }.$$
Observe that, by (\ref{eq-norm-wedge}),
$$C_{ k}^{^{\RR^d,{\cW_r}}}(\overline{w})= \sup_{n \geq k} {\Vert A_{n,k}\Vert \over 
	\vert A_{n,k} \overline{w}\vert_{_{\cW_r}}^{1/r} }\leq C^{1/r}\sup_{n \geq k} {\Vert\wedge^r A_{n,k}\Vert^{1/r} \over 
	\vert A_{n,k} \overline{w}\vert_{_{\cW_r}}^{1/r} }= C^{1/r} \left(C_{ k}^{^{\cW_r}}(\overline{w})\right)^{1/r} .$$
	
Proposition \ref{zekbghjkd} and \ref{prop-RCN}  enable to control the  moments of the  random variables $\ell_1^{^{^{(\rho), {\cW_r}}}}(\overline{W}_0)$ and  $ C_k^{^{\cW_r}}(\overline{W}_0)$ and ensure that, for  $\alpha<1/2$ and $\beta>3,$ 
 $$\mathbb E\left(\left(\ell_1^{^{(\rho), {\cW_r}}}(\overline{W}_0)\right)^\alpha\right)<+\infty\, \mathbb E\left((\ln^+ C_{ 1}^{^{\cW_r}}(\overline{W}_0))^\beta\right)<+\infty \ {\rm and} \ \mathbb E\left(\left(\ln^+ C_{ 1}^{^{\RR^d,{\cW_r}}}(\overline{W}_0)\right)^\beta\right)<+\infty.$$

We  conclude the proof  using the following criterion   which 
  slightly generalizes Criterion \ref{thm-from-RCN-to-cons}.
\end{proof}
\begin{criterion}\label{thm-from-RCN-to-cons-rep}
	Let $(\cW,|\cdot|_{_\cW})$ be a $\RR$-vector space and consider a representation  $\pi$ of ${\rm Gl}(d,\RR)$ on ${\rm Gl}(\cW)$. For simplicity, we write $Aw=\pi(A)w$ for any $w\in \cW$.
	
	Let $\nu_{_\cW}$ be a $\mub$-invariant probability measure on   the protective space ${\bf P}(\cW)$,    of $\cW$  and  let $ \overline{W}_0$  be  a random variable  with distribution   $\nu_{_\cW}$. 	
	
	For $\overline{w}\in \PP(\cW), 0<\rho<1$ and $r>0$, let
	\begin{equation}\label{RNCrandomvariable-rep} 
	\ell_1^{^{(\rho), \cW}}(\overline{w}):= \inf \{n \geq 1 \mid  \vert A_n \cdots A_1\overline{w}\vert_{_\cW} \leq \rho\}  
		\mbox{ and  }	
		C_{k}^{^{\RR^d,\cW}}(\overline{w}):= \sup_{n \geq k} {\Vert A_n\cdots A_{k}\Vert \over 
			\vert A_n\cdots A_{k} \overline{w}\vert_{_\cW}^{1/r} }
	\end{equation}
	
	Assume that   there exist   constants   $\alpha >0$, $   \beta >1 + {1\over \alpha}$  and $\gamma\geq \max\{{1\over \alpha},1\} $ such that the following hypotheses hold:
	
	${\bf A}_1^{^\cW}(\alpha)$-$\quad 
	\mathbb E\left(\left(\ell_1^{^{(\rho), \cW}}(\overline{W}_0)\right)^\alpha\right)<+\infty $ for any $\rho \in (0, 1)$.
	
	${\bf A}_2^{^\cW} (\beta) $-$  \quad   
	\mathbb E\left((\ln^+ C_{1}^{^{\RR^d,\cW}}(\overline{W}_0))^\beta\right)<+\infty  $ for some $r>0$. 	
	
	${\bf B} (\gamma)$-$\quad   
	\mathbb E\left((\ln^+\vert B_1\vert)^{\gamma}\right)<+\infty.
	$
	
	\noindent Then, for any $x \in \mathbb R^d$,   the chain $(X_n^x)_{n \geq 0}$ is recurrent.
\end{criterion}	
\noindent {\bf Proof of criterion \ref{thm-from-RCN-to-cons-rep}}.
The proof follows the same strategy as the one  of criterion \ref{thm-from-RCN-to-cons}. We present the main steps.

Let $(\overline{W}^{(k)})_{k \geq 0}$ be a sequence of independent and identically distributed random variables  with  values in ${\bf P}(\cW)$ and distribution $\nu_{_\cW}$. Assume that this sequence is independent of $(A_n, B_n)_{n \geq 1}$. As in the proof of  Criterion \ref{thm-from-RCN-to-cons}, we can suppose  $\overline{W}^{(0)}=\overline{W}_0$ without loss of generality. 
Fix $0<\rho<1$ such that $\ln \rho<-r\EE \left(\ln  C_{1}^{^{\RR^d,\cW}}(\overline{W}_0)\right).$
We set $\ell_0 =0$
and, for any $k \geq 1$, 
\[
\ell_k :=  \inf\{n \geq \ell_{k-1}+1: \vert A_n \cdots A_{\ell_{k-1}+1}  \overline{W}^{(k-1)} \vert_{_\cW} \leq \rho\}. 
\]
 Then,   the sub-process $(X_{\ell_n}^x)_{n \geq 0}$ is a multidimensional affine recursion corresponding to the  i.i.d. affine random transformations $\widetilde{g}_k$ with the same law as $\widetilde g_1$ given by 
$$
\widetilde{g}_1=(\widetilde{A_1}, \widetilde{B_1}):=(A_{\ell_{1}}\dots A_{1},\sum_{i= 1}^{\ell_{1}} A_{\ell_{1}} \cdots A_{i+1} B_i).$$
Let us now check that  the sub-process $(X_{\ell_n})_{n \geq 0}$ satisfies the hypotheses of Fact \ref{fact-cont-case}; this will prove that  $(X_{\ell_n})_{n \geq 0}$, hence $(X_n)_{n \geq 0}$, is recurrent. 

On the one hand, the Lyapunov exponent of the distribution $\tilde \mu$ of $\widetilde{g}_1$ is negative : 
$$\gamma_{\widetilde \mu}\leq \EE(\log \Vert A_{\ell_1} \cdots A_1\Vert) \leq \EE\left(\log(C_{ 1}^{^{\RR^d,\cW}}(\overline{W}_0) \vert A_{\ell_1} \cdots A_1 \overline{W}_0 \vert_{_\cW}^{1/r})\right) 
\leq \frac{\ln \rho}{r}+ \EE \left(\ln  C_{1}^{^{\RR^d,\cW}}(\overline{W}_0)\right)<0.$$

On the other hand, in order to check that $\widetilde{B_1}=B_{\ell_1,1}$ is log-integrable, we observe that 
\begin{align*}
	\vert B_{\ell_1,1}\vert  &\leq
	\sum_{i={1}}^{\ell_1} \Vert A_{\ell_1}\cdots A_{i+1}\Vert  \vert B_i\vert \leq \sum_{i=1}^{\ell_1} C_{i+1}^{^{\RR^d,\cW}}(\overline{W}_i) \vert B_i\vert \quad \mathbb P\text{-a.s.}
\end{align*}
where $\overline{W}_i:=A_{i,1}\cdot \overline{W}_0$. Indeed,    
\begin{align*}
	\Vert A_{\ell_1}\cdots A_{i+1}\Vert
	&\leq 
	C_{ i+1}^{^{\RR^d,\cW}}(\overline{W}_i)\vert A_{\ell_1}\cdots A_{i+1} \overline{W}_i \vert_{_\cW}^{1/r}
	\\
	&= 
	C_{ i+1}^{^{\RR^d,\cW}}(\overline{W}_i){\vert A_{\ell_1}\cdots A_{1} \overline{W}_0 \vert_{_\cW}^{1/r} \over \vert A_{i}\cdots A_{1} \overline{W}_0 \vert_{_\cW}^{1/r}}\leq 	C_{ i+1}^{^{\RR^d,\cW}}(\overline{W}_i){\rho^{1/r} \over \rho^{1/r} }= 	C_{ i+1}^{^{\RR^d,\cW}}(\overline{W}_i).
\end{align*}
Consequently,
\begin{equation*}
	\ln^+\vert B_{\ell_1,1}\vert 
	 \leq \ln^+\ell_1+\max_{1\leq i \leq \ell_1 } \ln^+C_{ i+1}^{^{\RR^d,\cW}}(\overline{W}_i) +\max_{1\leq i 
	 	\leq \ell_1} \ln^+ \vert B_i\vert\quad \mathbb P\text{-a.s.}
\end{equation*}
and we can conclude as in Criterion \ref{thm-from-RCN-to-cons}
by using Lemma \ref{lemmedusup}.

\rightline{$\Box$}
 \section*{Acknowledgements}

The authors would like to thank an anonymous referee for careful reading and remarks.

\end{document}